\documentclass[11pt]{amsart}

\usepackage[utf8]{inputenc}
\usepackage[T1]{fontenc}
\usepackage{lmodern}
\usepackage{float}
\usepackage[a4paper, total={6in, 8in}]{geometry}
\usepackage{stmaryrd, amsthm, mathtools, amssymb, dsfont, mathrsfs, bm, amsfonts, amsmath}
\usepackage{accents}
\usepackage{listings}
\usepackage{algorithm}
\usepackage{enumitem}
\usepackage{fullpage}
\usepackage{color}
\usepackage{todonotes}
\usepackage{booktabs}
\usepackage{caption}
\usepackage{subcaption}
\usepackage{hyperref}
\usepackage[noabbrev,capitalise]{cleveref}

\usepackage{xcolor}
\usepackage{graphicx}
\usepackage{setspace}
\usepackage{parskip}
\newtheorem{thm}{Theorem}
\newtheorem{prop}[thm]{Proposition}
\newtheorem{lemma}[thm]{Lemma}
\newtheorem{obs}[thm]{Observation}
\newtheorem{cor}[thm]{Corollary}
\newtheorem{claim}{Claim}[thm]
\newtheorem{question}{Question}
\newtheorem{conj}{Conjecture}
\theoremstyle{definition}

\crefname{defi}{Definition}{Definitions}
\theoremstyle{remark}

\newtheorem*{thm*}{Theorem}
\newtheorem*{lemma*}{Lemma}
\newtheorem*{cor*}{Corollary}
\newtheorem*{prop*}{Proposition}
\newtheorem*{defi*}{Definition}
\newtheorem*{rk*}{Remark}
\newtheorem*{conj*}{Conjecture}
\newtheorem*{fact*}{Fact}

\usepackage{pgfplots}
\pgfplotsset{compat=1.15}
\usepackage{mathrsfs}
\usetikzlibrary{arrows}

\usepackage[normalem]{ulem} 

\newcommand{\bb}[1]{\mathbb{#1}}
\newcommand{\mc}[1]{\mathcal{#1}}

\newcommand{\equa}[1]{\begin{equation}\nonumber #1\end{equation}}

\newcommand{\pth}[1]{\left( #1 \right)}
\newcommand{\floor}[1]{\left\lfloor #1 \right\rfloor}
\newcommand{\ceil}[1]{\left\lceil #1 \right\rceil}

\newcommand{\set}[1]{\left\{#1\right\}}

\newcommand{\card}[1]{\left| #1 \right|}

\newcommand{\NN}{\bb{N}}

\newcommand\restr[2]{{#1}_{|#2}}

\newcommand{\pr}[1]{\bb{P}\left[#1\right]}
\newcommand{\esp}[1]{\bb{E}\left[#1\right]}

\newcommand{\bigO}[2][]{\mc{O}_{#1}\!\pth{#2}}

\newcommand{\bigTheta}[2][]{\Theta_{#1}\!\pth{#2}}
\newcommand{\bigOmega}[2][]{\Omega_{#1}\!\pth{#2}}

\newcommand{\softOmega}[1]{\widetilde{\Omega}\pth{#1}}
\newcommand{\softO}[1]{\widetilde{\mc{O}}\!\pth{#1}}
\newcommand{\softTheta}[1]{\widetilde{\Theta}\pth{#1}}

\newcommand{\cA}{\mc{A}}

\newcommand{\cF}{\mc{F}}
\newcommand{\cG}{\mc{G}}

\newcommand{\F}{\mathscr{F}}

\newcommand{\eps}{\varepsilon}

\newcommand{\cycles}{\Omega}
\newcommand{\cycle}{\Pi}

\newcommand{\twice}{\Upsilon}
\newcommand{\twiceA}{\Lambda}
\newcommand{\twiceB}{\Phi}
\newcommand{\minusA}{\minuso}
\newcommand{\minusB}{\boxminus}
\newcommand{\boldc}{\mathbf{c}}

\newcommand{\boldL}{\mathbf{L}}
\newcommand{\boldX}{\mathbf{X}}

\DeclareMathOperator*{\argmin}{arg\,min}

\usepackage{version}

\newenvironment{subproof}[1][Proof of the Claim]{%
  \begin{proof}[#1]%
}{%
  \end{proof}%
}

\newcommand{\degmaxDeux}[2]{\Delta_{#1}(#2)}
\newcommand{\degmax}[1]{\degmaxDeux{2\ell}{#1}}
\newcommand{\bX}{\mathbf{X}}
\DeclareMathOperator{\polylog}{polylog}
\DeclareMathOperator{\ad}{ad}
\DeclareMathOperator{\ex}{ex}
\newcommand{\coDelta}{\Delta^{(2)}}
\newcommand{\cocoDelta}{\Delta^{(3)}}
\newcommand{\hammock}[1]{\langle #1 \rangle}
\newcommand{\pt}{3\pth{\frac{\log n}{n}}^{1/4}t^{-1/2}}

\title{Acyclic colourings of graphs with obstructions}

\author{Quentin Chuet}
\address{\'Equipe GALaC, LISN (Université Paris-Saclay),
Gif sur Yvette, France.}
\email{quentin.chuet@lisn.fr}

\author{Johanne Cohen}
\address{\'Equipe GALaC, LISN (Université Paris-Saclay),
Gif sur Yvette, France.}
\email{johanne.cohen@lisn.fr}

\author{François Pirot}
\address{\'Equipe GALaC, LISN (Université Paris-Saclay),
Gif sur Yvette, France.}
\email{francois.pirot@lisn.fr}

\begin{document}

\begin{abstract}
Given a graph $G$, a colouring of $G$  is \emph{acyclic} if it is a proper colouring of $G$ and every cycle contains at least three colours.
Its acyclic chromatic number $\chi_a(G)$ is the minimum~$k$ such that an acyclic $k$-colouring of $G$ exists.
When $G$ has maximum degree $\Delta$, it is known that $\chi_a(G) = \bigO{\Delta^{4/3}}$ as $\Delta \to \infty$, and that $\chi_a(G) = \bigO{\sqrt{t} \cdot \Delta}$ if in addition $G$ does not contain $K_{2,t}$ as a subgraph.
We study the extremal value of the acyclic chromatic number in the class of graphs of maximum degree $\Delta$ that do not contain some fixed subgraph $F$ on $t$ vertices.
We establish that this extremal value is at most $\bigO{t^{8/3}\Delta^{2/3}}$ if $F$ is a tree, $\bigO{\sqrt{t} \cdot \Delta}$ if $F$ is bipartite and can be made acyclic with the removal of one vertex, $2\Delta + \bigO{t\Delta^{2/3}}$ if $F$ is an even cycle of length at least $6$, and $\bigO{t^{1/4}\Delta^{5/4}}$ if $F=K_{3,t}$. 
Moreover, we exhibit an infinite family of obstructions $F$ that each induces a different asymptotic behaviour for this extremal value.
This is obtained with the derivation of lower bounds that come from the analysis of the acyclic chromatic number of a random graph drawn from either $G(n,p)$ or $G(n,n,p)$, that we entirely determine up to a $\polylog(n)$ factor.
As a byproduct, we can certify that most of our results are tight up to a $\Delta^{\mathcal{O}(1/t)}$ factor.
\end{abstract}

\maketitle

\section{Introduction}
Given a graph $G$ and an integer $k\ge 1$, an \emph{acyclic $k$-colouring} $\phi$ of $G$ is a proper $k$-colouring (a partition of $V(G)$ into $k$ independent sets which we call the \emph{colour classes} of $\phi$) such that each pair of colour classes induces a forest (or equivalently, every cycle contains at least three colours). The minimum integer $k$ such that an acyclic $k$-colouring of $G$ exists is the \emph{acyclic chromatic number of $G$}, which we denote $\chi_a(G)$.

Grunbaum~\cite{grunbaum1973acyclic} introduced the concept of acyclic colourings in 1973, and a few years later Kostochka~\cite{kostochka1978upper} proved that deciding whether $G$ has
an acyclic $3$-colouring is NP-complete.
Following this, many efforts have been made to either compute exactly the acyclic chromatic number for certain graphs~\cite{borodin1979acyclic} or to find good estimates for its value.

Fertin \emph{et al.}~\cite{fertin2003acyclic} have established a simple lower bound on the acyclic chromatic number: if $G$ has average degree $d$, then $\chi_a(G) > d/2 + 1$. 
Alon \emph{et al.}~\cite{AMR91} showed that there are graphs $G$ with maximum degree $\Delta$ for which $\chi_a(G)= \bigOmega{\frac{\Delta^{4/3}}{(\log \Delta)^{1/3}}}$ as $\Delta\to\infty$. 
This follows from an analysis of a random graph drawn from the Erdős–Rényi model $G(n,p)$ with $p\coloneqq \Theta((\log n/n)^{1/4})$.
 
Upper bounds on $\chi_a(G)$ in terms of the maximum degree $\Delta(G)$ have been studied for a few decades.
It has been established in 1991~\cite{AMR91} that, for every graph $G$ of maximum degree $\Delta$, one has $\chi_a(G) = \bigO{\Delta^{4/3}}$. Since then, there has been a couple of improvements on this bound~\cite{GMP20,ndreca2012improved,SeVo13}. To our knowledge, the state of the art is now as follows:  for every graph $G$ of maximum degree $\Delta$, one has $\chi_a(G) \le \frac{3}{2}\Delta^{4/3} + \bigO{\Delta}$. This bound was obtained with the so-called entropy compression method, developed following the algorithmic proof of the Lovász Local Lemma (LLL) due to Moser and Tardos~\cite{MoTa10}. It is tight up to a polylogarithmic factor, as certified by the lower bound described in~\cite{AMR91}. 
If moreover $G$ is a line-graph, then $\chi_a(G) = \bigTheta{\Delta}$~\cite{EsPa13}, and more generally for every $K_{2,t}$-free graph $G$ of maximum degree $\Delta$, $\chi_a(G) = \bigO{\sqrt{t} \cdot \Delta}$~\cite{AMR91,GMP20}.
We also note that the bounds in~\cite{GMP20} are constructive; in particular, there is a random algorithm that returns in expected polynomial time an acyclic $k$-colouring of any $d$-regular $K_{2,t}$-free graph $G$, where $k$ is within a factor $\bigO{\sqrt{t}}$ of the optimal.

Motivated by the above observation, we have sought obstructions that lead to a constant approximation algorithm for the acyclic colouring problem over $d$-regular graphs. We show that even cycles of length at most $d^{1/3}$ are such obstructions.
More generally, our work studies the impact of forbidding one or several fixed subgraphs $F$ in a graph $G$ on the value of $\chi_a(G)$. Given a fixed set of graphs $\F$, a graph G is called \emph{$\F$-free} if no graph $F\in \F$ appears as a (not necessarily induced) subgraph of $G$. When $\F=\{F\}$, we say that $G$ is $F$-free.
We let $a(d,\F) \coloneqq \max \{\chi_a(G) : G \mbox{ is $\F$-free and $\Delta(G)\le d$} \}$ be the extremal value of the acyclic chromatic number over $\F$-free graphs of maximum degree at most $d$. When $\F=\{F\}$ is a single-graph obstruction, we may use the notation $a(d,F)$.

This parallels an extensive study of the extremal value of the \emph{chromatic number} over $F$-free graphs of maximum degree at most $d$, which we denote $c(d,F)$. For a fixed obstruction $F$, up to a potential polyloglog factor, there are only two distinct asymptotic behaviours for $c(d,F)$, depending on whether $F$ is acyclic or not. If $F$ is a forest, then $c(d,F) = \bigO{1}$ since any $F$-free graph has bounded degeneracy (see Lemma~\ref{lem:tree-min-degree}) and therefore bounded chromatic number.
Otherwise (if $F$ contains at least one cycle), one has $\Omega(d/\log d) \le c(d,F) \le \bigO{d\log \log d/\log d}$. 
The lower bound follows from the fact that there are $d$-regular graphs of arbitrarily large girth and chromatic number at least $\frac{d}{2\ln d}$~\cite{bollobas1981independence}. The general upper bound reduces to the case $F=K_t$ for some fixed integer $t\ge 3$. It was first proven by Johansson~\cite{johansson1996choice}, and has been refined since then~\cite{molloy2019list,DJKS20+}.
The correct order of magnitude for $c(d,F)$ when $F$ is not acyclic is conjectured in~\cite{AKS99} to be $\Theta(d/\log d)$. They showed that the conjecture holds in the particular case that $F$ can be made bipartite by removing one vertex.

In contrast, our results indicate that $a(d,F)$ exhibits a more complex range of behaviours. We identify an infinite number of different regimes for $a(d,F)$; see Table~\ref{tab:bounds} for a detailed list of bounds on $a(d,F)$ with their references.

\begin{enumerate}[label=(\roman*)]
    \item If $F$ is the subdivision of a tree (or a subgraph thereof), then $a(d,F) = \bigO{1}$, cf.~\cite{choi2019characterization}.
    \item If $F$ is any other tree, then $\bigOmega{d^{1/2}} \le a(d,F) \le \bigO{d^{2/3}}$.
    \item If $F$ is bipartite and can be made acyclic by removing one vertex, then $a(d,F) = \Theta(d)$.
    \item If $F=K_{3,t}$, then there is $\eps_t \underset{t\to\infty}{\to} 0$ such that $\Omega\pth{d^{5/4-\eps_t}} \le a(d,K_{3,t}) \le \bigO{d^{5/4}}$.
    \item If $F=K_{4,t}$, then there is $\eps_t \underset{t\to\infty}{\to} 0$ such that $\Omega\pth{d^{4/3-\eps_t}} \le a(d,K_{4,t}) \le \bigO{d^{4/3}}$.
    \item For every rational of the form $\alpha = 1 + \frac{s-1}{2sk}$ with $s,k \ge 2$, and every $\eps > 0$, there exists a graph $F$ such that $\bigOmega{d^{\alpha - \eps}} \le a(d,F) \le \bigO{d^\alpha}$.
    \item If $F$ is not bipartite, or if $F$ has average degree at least $8$, then $ a(d,F) = \bigOmega{\frac{d^{4/3}}{(\log d)^{1/3}}}$.
\end{enumerate}

\subsection{Notations and terminology}
\label{sec:notations}
All graphs considered in this paper are finite and simple. 
Given a graph $G$, we respectively denote $\delta(G), \ad(G)$, and $\Delta(G)$ the minimum, average, and maximum degrees of $G$. Given two distinct vertices $u,v\in V(G)$, we let $N(u,v) \coloneqq N(u)\cap N(v)$ be their \emph{coneighbourhood}, and $\deg(u,v)\coloneqq |N(u,v)|$ be their \emph{codegree}. Given two disjoint subsets $X,Y\subseteq V(G)$, we denote $G[X,Y]$ the bipartite subgraph of $G$ with vertex-set $X\cup Y$ and edge-set $E(G) \cap (X\times Y)$.

Given an integer $n\ge 1$, $P_n$ and $C_n$ respectively denote the path and cycle graphs on $n$ vertices. For a connected bipartite graph $H$ of parts $A$ and $B$, we define the \emph{hammock} of $H$, denoted $\hammock{H}$, as a copy of $H$ with two additional vertices $a$ and $b$ (which we call \emph{anchor vertices}) such that $N(a) = A$ and $N(b) = B$. For $k\ge 2$, the \emph{$k$-subdivision} of a graph $H$, denoted $H^{1/k}$, is obtained by replacing each edge of $H$ with a path of length $k$. The \emph{subdivision} of $H$ is $H^{1/2}$. A \emph{subdivided tree} is a tree obtained by taking the subgraph of a tree subdivision; in other words, it is a tree with an even distance between every pair of vertices of degree at least $3$.

We denote by $[n]$ the set of integers $\{0, \ldots, n-1\}$. Given a set of elements $S$, and a multiset $\F$ of subsets of $S$, we denote by $\Delta_q(\F)$ the maximum over all $x\in S$ of the number of members of $\F$ of cardinality $q$ that contain $x$. For example,
given a set of cycles $\Pi$ in a graph $G$, $\degmax{\Pi}$ corresponds to the maximum number of cycles of length $2\ell$ in $\Pi$ that intersect in a common vertex $v \in V(G)$ --- here we abuse the notation by referring directly to a cycle $C$ instead of its vertex-set $V(C)$.

In order to ignore polylogarithmic terms in our bounds, we shall often use the \emph{soft-$\mc{O}$} notation: for every real-valued functions $f$ and $g$, we write $f(x) = \softO{g(x)}$ if $f(x) = \bigO{g(x) \log(x)^\alpha}$, for some constant $\alpha \in \mathbb{R}$. The \emph{soft-$\Omega$} notation is defined analogously, and we write $f(x) = \softTheta{g(x)}$ if $f(x) = \softO{g(x)}$ and $f(x) = \softOmega{g(x)}$. For example, $x^2/\log(x) = \softTheta{x^2}$.

When we consider two variables $x$ and $t$ simultaneously, we may write $f(x) = \bigO[t]{g(x)}$ to indicate that $f(x) = \bigO{h(t) \cdot g(x)}$ for some absolute function $h$ which only depends on $t$, and the notations $\Omega_t$ and $\Theta_t$ are defined in the same manner. For example, $2^t x^2 = \bigTheta[t]{x^2}$

Throughout this paper, when we say that a property holds \emph{with high probability} (w.h.p.) for a random graph $G$ on $n$ vertices, we mean with probability tending to $1$ as $n\to\infty$.  

\subsection{Organisation of the document}  
Our work mainly consists in finding bounds for $a(d,\F)$, with a specific focus on the case where $\F=\{F\}$ is a single-graph obstruction. We summarise these bounds in Table~\ref{tab:bounds}.

\renewcommand{\arraystretch}{1.5}
\begin{table}[!ht]
    \centering
    \begin{tabular}{l c c l}
    \toprule
    $\F$ & Upper bound & Lower bound & Reference\\
    \midrule 
    $\emptyset$ & $\frac{3}{2} d^{4/3} + \bigO{d} $ & $\Omega\pth{\frac{d^{4/3}}{(\log d)^{1/3}}}$ &~\cite{GMP20,AMR91}\\
    $\{K_{2,t}\}$ & $\bigO{\sqrt{t} \cdot d}$  & $d/2$ &~\cite{AMR91}, Prop.~\ref{prop:linearlowerbound:cycle}\\
    $\{K_{3,t}\}$ & $\bigO{t^{1/4} d^{5/4}} $ & $\Omega\pth{d^{5/4-\eps_t}}$ & Cor.~\ref{cor:2acyclic}, Cor.~\ref{cor:K_3t-free}\\
    $\{K_{4,t}\}$ & & $\Omega\pth{d^{4/3-\eps_t}}$ & Cor.~\ref{cor:K_4t-free}\\
    $\{\,\hammock{K_{s,t}^{1/k}}\,\}$ & $\bigO[s,t,k]{d^{1 + \frac{s-1}{2sk}}}$ &  $\bigOmega{d^{1 + \frac{s-1}{2sk} - \eps_t}}$ & Cor.~\ref{cor:subdivision-upperbound}, Cor.~\ref{cor:subdivision-lowerbound} \\
    $\{T\}\colon$  subdivided tree & constant $C_T$ & &\cite{choi2019characterization} \\
    $\{C_4,T\}$ where $T\colon$\,tree, size $t$ & $\bigO{\sqrt{t^5d}}$ & $\sqrt{d/2}$ & Cor.~\ref{cor:C4-F-free}, Prop.~\ref{prop:not-subdivided}  \\
    $\{T\}$ where $T\colon$\,tree, size $t$ & $\bigO{t^{8/3}d^{2/3}}$ & $\sqrt{d/2}$ & Cor.~\ref{cor:forest}, Prop.~\ref{prop:not-subdivided}  \\
    $\{C_4\}$ & $2.7627\,d$  & $d/2$ & Th.~\ref{thm:c4free}, Prop.~\ref{prop:linearlowerbound:cycle}   \\
    $\{C_{2t}\}$ for some $3\le t \le d^{1/3}$ & $2d + \bigO{t d^{2/3}}$ & $d/2$ & Th.~\ref{thm:C2t}, Prop.~\ref{prop:linearlowerbound:cycle}  \\
    $\{C_3,C_4,C_5,C_6\}$ & $1.7633\,d + \bigO{\sqrt{d}}$  & $d/2$ & Th.~\ref{thm:girth7}, Prop.~\ref{prop:linearlowerbound:cycle} \\  
    \bottomrule
    \end{tabular}
    \caption{Our bounds on $a(d,\F)$ with various obstructions. The $\bigO{\cdot}$ and $\Omega(\cdot)$ terms are considered as $d\to \infty$. For $K_{3,t}$, $K_{4,t}$ and $\hammock{K_{s,t}^{1/k}}$, we have $\varepsilon_t \rightarrow 0$ as $t\rightarrow \infty$. For $\hammock{K_{s,t}^{1/k}}$ specifically, we assume $s,k \ge 2$ and $t\ge s$.}
    \label{tab:bounds}
\end{table}

In Section~\ref{sec:framework}, we present a technical theorem (Theorem~\ref{thm:main}) that certifies the existence of \mbox{$K$-colourings} with specific constraints. Applications of
Theorem~\ref{thm:main} require that we derive upper bounds on the number of cycles with specific properties in $\F$-free graphs. In Section~\ref{sec:general}, we provide a general estimate based on the \emph{extremal number $\ex(n,F)$}, from which we derive that
\begin{equation} 
\label{eq:turan}
a(d,\hammock{F}) = \bigO{\sqrt{d \cdot \ex(d,F)}}
\end{equation}
for every connected bipartite graph $F$, where we recall that $\hammock{F}$ is the hammock of $F$ as introduced in Subsection~\ref{sec:notations}. This is general enough to retrieve the bound $a(d,K_{2,t}) = \bigO{\sqrt{t} \cdot d}$ mentioned earlier, and to derive many other similar ones.

In Section~\ref{sec:lowerbounds}, we first observe that we may restrict ourselves to the analysis of connected obstructions.
We then evaluate the tightness of the upper bounds derived in Section~\ref{sec:general} through a thorough analysis of $\chi_a(G)$ when $G$ is a random bipartite graph drawn from $G(n,n,p)$ --- that is a random subgraph of $K_{n,n}$ where each edge is kept independently with probability $p$ --- for $p$ that ranges roughly from $n^{-1/2}$ to $n^{-1/4}$. We determine a lower bound that holds with high probability on $\chi_a(G)$, which we later show to be tight up to a $\polylog(n)$-factor. Using this lower bound, we show that the upper bound on $a(d,F)$ that we can derive from \eqref{eq:turan} for many specific graphs $F$ on $t$ vertices is tight up to a $\Delta^{\bigO{1/t}}$-factor.
We observe that if $F$ is not bipartite then $a(d,F) = \softTheta{d^{4/3}}$, and if $F$ is not acyclic then the lower bound of Fertin \emph{et al.}~\cite{fertin2003acyclic} implies that $a(d,F)>d/2$ (cf. Proposition~\ref{prop:linearlowerbound:cycle}).

Following that observation, we focus in Section~\ref{sec:forests} on the case where $\F$ contains a tree. By a result of Choi \emph{et al.}~\cite{choi2019characterization}, $a(d,F)$ is bounded by a constant (as $d\rightarrow \infty$) if and only if $F$ is a subdivided tree. We show that if $F$ is any other tree, then $a(d,\{C_4,F\}) = \bigTheta[F]{\sqrt{d}}$ and $\sqrt{d/2} \le a(d,F) \le \bigO[F]{d^{2/3}}$.

We finish in Section~\ref{sec:cycle} by focusing on cycle obstructions. When $F$ is an even cycle, we determine $a(d,F)$ up to an absolute multiplicative constant, and we tighten this as much as possible for the specific obstruction $\F=\{C_3,C_4,C_5,C_6\}$ (i.e. we analyse the extremal value of $\chi_a$ over the class of graphs of girth $7$ and bounded maximum degree).

We conclude in Section~\ref{sec:conclusion} by tying all our bounds together to obtain a good estimation (up to a $\polylog(n)$-factor) of $\chi_a(G)$ for a random graph $G$ drawn from $G(n,p)$ or $G(n,n,p)$, for every value of $p=p(n)\in (0,1)$ (cf. Theorem~\ref{thm:randomgraph-generalbounds}). We end by proposing a few open problems.

\section{The general framework} \label{sec:framework}

The proof of our main theorem is an abstract generalisation of that of~\cite[Theorem~2]{GMP20}, where we have replaced the compression entropy machinery with an inductive counting approach, sometimes referred to as \emph{Rosenfeld counting} following its first use in the context of graph colouring in~\cite{rosenfeld2020another}. 
We also note that it is a special case of~\cite[Theorem~3]{WaWo22}, although we have decided to include the relatively short proof for completeness and to avoid relying on an abstract framework on hypergraphs.
As a byproduct, we obtain exponential lower bounds for the number of colourings satisfying the hypothesis of the theorem.
We note, however, that it would be possible to do the same using entropy compression and obtain a random algorithm that returns such a colouring in expected polynomial time.

We begin by introducing a few notions needed to formulate the theorem, which requires a high level of generality so that we can use it in the variety of applications appearing in the forthcoming sections.

Given a graph $G$, a set of \emph{constraints} $\Gamma$ consists of a set of pairs of vertices of $G$. For a given vertex $v\in V(G)$, we let $N_\Gamma(v)$ consist of all vertices that form a pair in $\Gamma$ together with $v$. The \emph{constraint-degree of $v$} is $\deg_\Gamma(v) \coloneqq |N_\Gamma(v)|$, and the \emph{maximum degree of $\Gamma$} is $\Delta(\Gamma) \coloneqq \max_{v\in V(G)} \deg_\Gamma(v)$.
Given an integer $k\ge 1$, a \emph{$\Gamma$-proper $k$-colouring of $G$} is a mapping $\phi\colon V(G) \to [k]$ such that $\phi(u) \neq \phi(v)$ for every $\{u,v\} \in \Gamma$. For instance, an $E(G)$-proper $k$-colouring of $G$ coincides with the usual notion of proper $k$-colouring. More generally, letting $H$ be the graph on vertex-set $V(G)$ and edge-set $\Gamma$, a $\Gamma$-proper $k$-colouring of $G$ is a proper $k$-colouring of $H$.

Given an even cycle $C=(v_0,\ldots,v_{2\ell-1})$ in $G$ and a $k$-colouring $\phi$ of $G$ (which may be improper), we say that $C$ is \emph{bicoloured} in $\phi$ if $\phi(v_i)=\phi(v_{i+2 \bmod 2\ell})$ for every $i\in [2\ell]$.
We say that $C$ is \emph{$\Gamma$-free} if $\{v_i,v_{i+2 \bmod 2\ell}\} \notin \Gamma$ for every $i\in [2\ell]$. 
Given a set $\Pi$ of cycles in $G$, we say that $\phi$ is \emph{$\Pi$-acyclic} if no cycle in $\Pi$ is bicoloured in $\phi$.
We make the following observation.

\begin{prop}
\label{prop:gamma-free-cycles}
Let $G$ be a graph, let $\Gamma$ be a set of constraints on $G$, and let $\Pi_0 \supseteq \Pi_1$ be two sets of cycles of $G$. 
If $\Pi_1$ is the set of $\Gamma$-free cycles of $\Pi_0$, then any $\Gamma$-proper $\Pi_1$-acyclic colouring of $G$ is also a $\Gamma$-proper $\Pi_0$-acyclic colouring of $G$.
\end{prop}

\begin{proof}
    Let $\phi$ be a $\Gamma$-proper $\Pi_1$-acyclic colouring of $G$. Assume for the sake of contradiction that some cycle $C=(v_0, \ldots, v_{2\ell-1}) \in \Pi_0$ is bicoloured in $\phi$. Then $C \notin \Pi_1$, so $C$ is not $\Gamma$-free. Hence there exists $i$ such that $\{v_i, v_{i+2 \bmod 2\ell}\} \in \Gamma$. Since $C$ is bicoloured, we have $\phi(v_i)=\phi(v_{i+2\bmod 2\ell})$; this contradicts the fact that $\phi$ is $\Gamma$-proper.
\end{proof}

We are now ready to state our main technical theorem.

\begin{thm}
\label{thm:main}
Let $G$ be a graph, $\Gamma$ a set of constraints, and $\Pi$ a set of cycles of $G$.
Fix $\tau \ge 1$, and
\[ K \coloneqq \Delta(\Gamma) + \tau + \sum_{\ell\ge 2}\limits \frac{\degmax{\Pi}}{\tau^{2\ell-3}}.\]
Then there exist at least $\tau^{|V(G)|}$ $\Gamma$-proper $\Pi$-acyclic $\ceil{K}$-colourings of $G$.
\end{thm}

\begin{proof}
For every subgraph $H \subseteq G$ and every family of subsets of vertices $X$ of $G$, we denote $X[H] \coloneqq X\cap 2^{V(H)}$. We let $\cA(H)$ be the set of $\Gamma[H]$-proper $\Pi[H]$-acyclic $\ceil{K}$-colourings of $H$. By convention, if $H$ is the empty graph, $\cA(H)$ contains a single trivial colouring. We show with a strong induction that, for every subgraph $H \subseteq G$,
\begin{equation}
\label{eq:IH}
\tag{IH\ref*{thm:main}}
 \forall v_0 \in V(H), \quad \card{\cA(H)}\ge \tau \card{\cA(H-v_0)}.
\end{equation}
\noindent
If $V(H)$ is empty, \eqref{eq:IH} is trivially true. Suppose $V(H) \neq \emptyset$ and let $v_0 \in V(H)$. By induction, assume \eqref{eq:IH} is true for every strict subgraph $H' \subset H$. We define $\cF$, the set of \emph{flawed extensions}, as the subset of $\ceil{K}$-colourings of $H$ such that every colouring $\phi \in \cF$ is $\Gamma[H-v_0]$-proper $\Pi[H-v_0]$-acyclic, but contains a conflict due to the colour given to $v_0$. In symbols, $\phi\in \cF$ if and only if $\restr{\phi}{H - v_0} \in \cA(H-v_0)$ and $\phi \notin \cA(H)$.
We distinguish two types of conflict.
\begin{enumerate}[label=(\alph*)]
	\item There is a vertex $u \in V(H)$ such that $\{u,v_0\}\in \Gamma$ and $\phi(u) = \phi(v_0)$.
	\item $v_0$ is contained in a bicoloured cycle of $\Pi$.
\end{enumerate}
\noindent
We let $\cF_a$ and $\cF_b$ be the subsets of colourings that respectively contain a conflict of type (a) or (b). These two subsets are not necessarily disjoint.

By definition of $\cF$, we have have $\card{\cA(H)} = K\card{\cA(H-v_0)} - \card{\cF}$, and $\card{\cF} \le \card{\cF_a} + \card{\cF_b}$.

\noindent
To complete the proof of the induction, we claim that 
\begin{enumerate}[label=(\roman*)]
	\item \label{it:(i)}  $\displaystyle \; \card{\cF_a} \le \Delta(\Gamma) \card{\cA(H-v_0)}$;
	\item \label{it:(ii)}  $\displaystyle \; \card{\cF_b} \le \card{\cA(H-v_0)} \sum_{\ell\ge 2} \frac{\degmax{\Pi}}{\tau^{2\ell-3}} $.
\end{enumerate}

\begin{subproof}[Proof of \ref{it:(i)}]
Let $c \in \cA(H-v_0)$. By definition, $v_0$ belongs to at most $\Delta(\Gamma)$ constraints in $\Gamma$, hence there are at most $\Delta(\Gamma)$ flawed extensions of $c$ that induce a conflict of type $(a)$. Therefore, we have $\card{\cF_a} \le \Delta(\Gamma) \card{\cA(H-v_0)}$.
\end{subproof}

\begin{subproof}[Proof of \ref{it:(ii)}]

For $\ell \ge 2$, let $\cycles_{2\ell}$ be the set of cycles in $\Pi$ of length $2\ell$ which contain $v_0$ (we define it only for even cycles since we never care about odd cycles). By assumption, $|\cycles_{2\ell}|\le \degmax{\Pi}$. 
For a cycle $C \in \cycles_{2\ell}$, we let $\cF_C \subseteq \cF_b$ be the subset of colourings for which $C$ is effectively bicoloured. Clearly, we have $\cF_b =\bigcup_{\ell \ge 2} \bigcup_{C \in \cycles_{2\ell}} \cF_C$.

Let $C = (v_0,\ldots,v_{2\ell-1}) \in \cycles_{2\ell}$. We construct an injection from $\cF_C$ to $\cA(H \setminus \{v_0,\ldots,v_{2\ell-3}\})$. Let $\phi \colon \cF_C \to \cA(H \setminus \{v_0,\ldots,v_{2\ell-3}\})$ be the application that simply uncolours $v_0,\ldots,v_{2\ell-3}$ (uncolouring $v_0$ ensures that all conflicts are resolved). The inverse application $\phi^{-1}$ colours $v_0,\ldots,v_{2\ell-3}$ by alternating the colours of $v_{2\ell-1}$ and $v_{2\ell-2}$, which by definition is the only way to ensure that $C$ is bicoloured. Therefore, $\card{\cF_C} \le \card{\cA(H \setminus \{v_0,\ldots,v_{2\ell-3}\})}$. 

\noindent
We apply the induction hypothesis \eqref{eq:IH} iteratively on the vertices $v_{2\ell-3},\ldots,v_1$ in that order, and obtain that
\begin{align*}
\card{\cF_C} \le \card{\cA(H \setminus \{v_0,\ldots,v_{2\ell-3}\})} \le \frac{1}{\tau}\card{\cA(H \setminus \{v_0,\ldots,v_{2\ell-4}\})} \le \cdots \le  \frac{1}{\tau^{2\ell - 3}}\card{\cA(H-v_0)}.
\end{align*}
\noindent
Finally, since  $\degmax{\Pi}$ corresponds to the maximum number of cycles $C\in \Pi$ of length $2\ell$ that contain a fixed vertex $v_0 \in V(G)$,  we conclude that
\equa{
\begin{split}
\card{\cF_b} &\le \sum_{\ell \ge 2}\sum_{C \in \cycles_{2\ell}} \card{\cF_C} \le \sum_{\ell \ge 2}\sum_{C \in \cycles_{2\ell}} \frac{1}{\tau^{2\ell-3}}\card{\cA(H-v_0)}
\le \sum_{\ell \ge 2} \frac{\degmax{\Pi}}{\tau^{2\ell-3}}\card{\cA(H-v_0)}.
\qedhere
\end{split}
}

\end{subproof}
\noindent
A straightforward application of \ref{it:(i)} and \ref{it:(ii)}completes the proof of the induction:
\begin{align*}
\card{\cA(H)} &\ge K\card{\cA(H-v_0)} - \card{\cF_a} - \card{\cF_b} \\
&\ge \pth{K - \Delta(\Gamma) - \sum_{\ell\ge 2} \frac{\degmax{\Pi}}{\tau^{2\ell-3}}}\card{\cA(H-v_0)}\\
&\ge \tau \card{\cA(H-v_0)}.
\end{align*}
\noindent
An iterative application of \eqref{eq:IH} to every $v\in V(G)$ implies that $\card{\cA(G)} \ge \tau^{\card{V(G)}} > 1$.
\end{proof}

\section{A general upper bound}\label{sec:general}
In this section, we present a general application of Theorem~\ref{thm:main} that leads to non-trivial upper bounds on $a(d,F)$ for a large family of bipartite graphs $F$. For an integer $n$ and a graph $F$, the \emph{extremal number of $F$}, denoted $\ex(n,F)$, is the maximum number of edges in a $F$-free graph with $n$ vertices (this is also sometimes called the \emph{Tur\'an number} or \emph{Tur\'an function}). For a survey on this topic, see~\cite{furedi2013history}.

We recall that for a connected bipartite graph $F$ of parts $A$ and $B$, we define the \emph{hammock} of $F$, denoted $\hammock{F}$, as a copy of $F$ with two additional vertices $a$ and $b$ (called \emph{anchor vertices}) such that $N(a)=A$ and $N(b) = B$. The next lemma concerns the number of cycles in a $\hammock{F}$-free graph. A similar (albeit weaker) lemma was used by Bernshteyn~\cite{bernshteyn2016new} to bound the acyclic chromatic index of graphs with bipartite obstructions. 

\begin{lemma}\label{lem:turan-cycles}
    Let $F$ be a connected bipartite graph. Given any $\hammock{F}$-free graph $G$ of maximum degree $\Delta$, and for every $\ell \ge 2$, the number of cycles of length $2\ell$ that intersect in some common vertex $u\in V(G)$ is
   $$ \degmax{\Omega_G} \le \Delta^{2\ell - 3}\ex(2\Delta, F).$$
\end{lemma}
\begin{proof}
Let $u \in V(G)$. There are at most $\Delta^{2\ell-3}$ paths of length $2\ell-3$ starting from $u$. Let $v$ be the other endpoint of such a path, and let us denote $p_3(u,v)$ the number of paths of length $3$ from $u$ to $v$; we claim that 
\begin{equation}
\label{eq:claim}
    p_3(u,v) \le 2\ex(2\Delta,F).
\end{equation}
By concatenating a path of length $2\ell-3$ starting at $u$ and ending in some vertex $v$, and a path of length $3$ from $u$ to $v$, we obtain a cycle of length $2\ell$ containing $u$. Moreover, given any such cycle, there are two ways to obtain it in this way, depending on the direction in which we traverse it. Therefore the number of cycles of length $2\ell$ going through $u$ is at most $\frac{1}{2} \Delta^{2\ell-3} \cdot 2\ex(2\Delta, F)$, as desired.
We now prove the claim.

Let $X \coloneqq N(u) \cap N(v)$, $A \coloneqq N(u) \setminus X$, and $B\coloneqq N(v) \setminus X$. From each edge $xy \in G[X]$ one can construct exactly two paths of length $3$ from $u$ to $v$, namely $u-x-y-v$ and $u-y-x-v$. From each edge in $G[A,B]$, $G[A,X]$, and $G[X,B]$, one can construct one path of length $3$ from $u$ to $v$. This covers all possible ways of constructing a path of length $3$ from $u$ to $v$, thus
\begin{equation}\label{eq:p3}
    p_3(u,v) = 2|E(G[X])| + |E(G[A,B])| + |E(G[A,X])| + |E(G[X,B])|.
\end{equation}

We observe that $G[X]$, $G[A,B]$, $G[A,X]$, and $G[X,B]$ do not contain a copy of $F$, otherwise $u$ and $v$ could act as anchor vertices to retrieve $\hammock{F}$ as a subgraph of $G$ --- while this is clear in the case of $G[X]$, for the three other cases we highly rely on the fact that $F$ is connected, and therefore its bipartition is unique and must agree with the one of the bipartite graph we are considering. Let $x \coloneqq |X|$, and observe that $|A|,|B| \le \Delta-x$. We infer that
\begin{equation}
\label{eq:turan-hammock}
    p_3(u,v) \le 2\ex(x,F) + \ex(2(\Delta-x), F) + \ex(\Delta-x+x,F) + \ex(x + \Delta-x,F).
\end{equation}

The extremal number is superadditive: we have $\ex(n,F) + \ex(m,F) \le \ex(n+m,F)$ for every $n,m \ge 0$. To see this, let $H_n$ (resp. $H_m$) be an $F$-free graph with $n$ (resp. $m$) vertices and $\ex(n,F)$ (resp. $\ex(m,F)$) edges. Then the disjoint union $H_n \cup H_m$ is a $F$-free graph with $n+m$ vertices and $\ex(n,F) + \ex(m,F)$ edges. Using this observation, we infer that
$$p_3(u,v) \le \ex(2x,F) + \ex(2\Delta - 2x,F) + 2\ex(\Delta,F) \le 2\ex(2\Delta,F),$$
which proves our claim.
\end{proof}

\begin{thm}\label{thm:upperbound-general}
    Let $F$ be a connected bipartite graph with at least $4$ vertices. Then for $d \ge 3$,
    $$a(d,\hammock{F}) \le \ceil{2\sqrt{d \cdot \ex(2d,F)} + 3d}.$$
\end{thm}

\begin{proof}
Let $G$ be an $\hammock{F}$-free graph of maximum degree $\Delta$. Let $\lambda \coloneqq \ex(2\Delta, F)$. We prove Theorem~\ref{thm:upperbound-general} through an application of Theorem~\ref{thm:main} with $\Gamma=E(G)$ and $\Pi=\Omega_G$. We fix $\tau \coloneqq \sqrt{\Delta \lambda}$, so that there exist $\tau^{\card{V(G)}}$ proper acyclic $\ceil{K}$-colourings of $G$, with
\begin{align*}
    K &= \Delta + \tau + \sum_{\ell\ge 2} \frac{\degmax{\Omega_G}}{\tau^{2\ell-3}} \\
    & \le \Delta + \sqrt{\Delta \lambda} + \sum_{\ell\ge 2} \frac{\Delta^{2\ell - 3}\,\lambda}{\pth{\Delta \lambda}^{\ell-3/2}}& \mbox{by Lemma~\ref{lem:turan-cycles};}\\
    & = \Delta + 2\sqrt{\Delta \lambda} + \sum_{\ell\ge 3} \Delta^{\ell  - 3/2}\,\lambda^{5/2 -\ell}.
\end{align*}
As $F$ is a connected graph with at least $4$ vertices and $\Delta\ge 3$,  $\lambda = \ex(2\Delta, F) \ge 2\Delta-1 \ge \frac{5}{3}\Delta$. This is certified by the graph consisting of $\floor{\frac{2\Delta}{3}}$ disjoint triangles and a disjoint $K_{2\Delta \bmod 3}$. As a result, for $\ell \ge 3$, the term $\Delta^{\ell  - 3/2}\,\lambda^{5/2 -\ell}$ is bounded from above by $\Delta \cdot \pth{5/3}^{5/2 - \ell}$; therefore the sum in the above equality is bounded from above by the geometric series $$\sum_{\ell\ge 3} \Delta \pth{\frac{5}{3}}^{5/2 - \ell} = \frac{\sqrt{15}}{2}\, \Delta < 2\Delta.$$
It follows from the previous computations that $K \le 2\sqrt{\Delta \lambda} \, + 3\Delta$.
\end{proof}

\subsection{Corollaries for single-graph obstructions}\label{sec:corollaries}

We shall make use of the following simple observation to relate $a(d,F)$ to $a(d,\hammock{H})$ when $F \subseteq \hammock{H}$.

\begin{obs}\label{obs:upperbound-subgraph}
    Suppose $F \subseteq F'$. Then for all $d\ge 0$, $a(d,F) \leq a(d,F')$.
\end{obs}

In order to use Theorem~\ref{thm:upperbound-general}, we need upper bounds on the extremal number $\ex(n,F)$ for various connected bipartite graphs $F$. If $\ex(n,F) = \bigO[F]{n}$, then we obtain a linear upper bound on $a(d,\hammock{F})$, whereas if $\ex(n,F) = \Omega(n^{5/3})$, then we do not improve upon the general upper bound $a(d,\hammock{F}) = \mathcal{O}(d^{4/3})$.

Observe that if $F$ is a connected graph with at least $t \ge 3$ vertices and $n$ is a multiple of $t-1$, then $\ex(n,F) \ge \frac{t-2}{2}n$, as certified by the graph consisting of $n/(t-1)$ disjoint copies of $K_{t-1}$. Erd\H os and Gallai~\cite{ErGa59} proved that this lower bound is tight for $P_t$, the path on $t$ vertices.

\begin{thm}[Erd\H os, Gallai, 1959]
\label{thm:turan-paths}
Let $t\ge 2$. Then $\ex(n,P_t) \le \frac{t-2}{2}n$.
\end{thm}

Theorem~\ref{thm:turan-paths} has been conjectured to hold more generally when $P_t$ is replaced with any tree on $t$ vertices. 

\begin{conj}[Erd\H os, S\' os, 1963]
    Let $T$ be a tree on $t\ge 2$ vertices. Then $\ex(n,T) \le \frac{t-2}{2}n$.
\end{conj}

Ajtai, Komlós, Simonovits, and Szemerédi have reportedly proved the Erd\H os-S\' os Conjecture when $t$ is large enough. Their work is still in preparation for publication. For our purposes, we shall use a well-known $2$-approximation of the Erd\H os-S\' os Conjecture. We include a simple proof relying on two classical lemmas, which we will reuse later.

\begin{lemma}[Folklore] \label{lem:average-min-degree}
Every multigraph of average degree $d>0$ contains a sub-multigraph of minimum degree at least $\floor{d/2} + 1$.
\end{lemma}

\begin{lemma}
\label{lem:tree-min-degree}
Let $T$ be any fixed rooted tree on $t$ vertices. If a graph $G$ has minimum degree $t-1$, then a copy of $T$ is rooted in each vertex $v\in V(G)$ in $G$.
\end{lemma}

\begin{proof}
The proof is standard; we include it for completeness.
We construct a copy of $T$ rooted in $v$ greedily, by following a depth-first search ordering $v_1, \ldots, v_t$ 
of $V(T)$. 
We first fix $T_1 \coloneqq \{v\}$, then for every $2\le i\le t$ we construct $T_i$ by adding $v_i$ to  the tree $T_{i-1}$ already constructed, so that $T_i$ is isomorphic to $T[\{v_1, \ldots, v_i\}]$. To do so, we need to find a neighbour of the parent node $v_j$ of $v_i$ in $T$ (by assumption, we have $j<i$) that does not belong to $V(T_{i-1})$.
The existence is due to the fact that $v_j$ has at most $i-2 < t-1$ neighbours in $T_{i-1}$ and at least $t-1$ neighbours in $G$. At the end of that process, we obtain a tree $T_t$ rooted in $v$ and isomorphic to $T$.
\end{proof}

\begin{cor}\label{cor:turan-tree}
    Let $T$ be a tree with $t \ge 3$ vertices. Then $\ex(n,T) < (t-2)n$.
\end{cor}
\begin{proof}
    Let $G$ be a $T$-free graph with $n$ vertices. Suppose by contradiction that $G$ has at least $(t-2)n$ edges, i.e. $G$ has average degree at least $2(t-2)$. By Lemma~\ref{lem:average-min-degree}, $G$ contains a subgraph of minimum degree at least $t-1$. By Lemma~\ref{lem:tree-min-degree}, $H$ contains a copy of $T$, a contradiction.
\end{proof}

From Corollary~\ref{cor:turan-tree}, one can immediately deduce a linear upper bound on $a(d,\hammock{T})$ assuming $T$ is a tree, by a direct application of Theorem~\ref{thm:upperbound-general}.

\begin{cor}\label{cor:hammock-tree}
    Let $T$ be a tree with $t$ vertices. Then
    $$a(d,\hammock{T}) = \bigO{\sqrt{t} \cdot d}.$$
\end{cor}

While hammocks of trees can be considered   an \emph{ad-hoc} family of graphs, its closure under the subgraph operation contains several interesting usual graph classes. We now state a weaker corollary which concerns bipartite obstructions that become acyclic after removing one vertex.

\begin{cor}\label{cor:1acyclic}
    Let $F$ be a connected bipartite graph with $t$ vertices such that $F-v$ is acyclic for some $v\in V(F)$. Then
    $$a(d,F) = \bigO{\sqrt{t} \cdot d}.$$
\end{cor}
\begin{proof}
    Consider the components of $F-v$, and select one edge per component which connects it to $v$. Define $T$ by taking a copy of $F$ and removing all edges incident to $v$ which were not selected: one easily verifies that $T$ is a spanning tree of $F$ and that $F \subseteq \hammock{T}$. By Observation~\ref{obs:upperbound-subgraph} and Corollary~\ref{cor:hammock-tree}, we conclude that $a(d,F) \le a(d,\hammock{T}) = \bigO{\sqrt{t} \cdot d}$.
\end{proof}

Among other obstructions, Corollary~\ref{cor:1acyclic} notably concerns the even cycle $C_{2t}$, which will be treated separately in Section~\ref{sec:cycle}, and the complete bipartite graph $K_{2,t}$, which was already studied in~\cite{AMR91} and~\cite{GMP20}. We now investigate what happens when the obstruction is a larger complete bipartite subgraph. To that end, we will rely on the celebrated K\H{o}v\'{a}ri-S\'{o}s-Tur\'{a}n Theorem~\cite{kHovari1954problem} which provides an 
upper bound on $\ex(n,K_{s,t})$.

\begin{thm}[K\H{o}v\'{a}ri, S\'{o}s, Tur\'{a}n, 1954]
    Let $s \le t$ be positive integers. Then
    $$\ex(n,K_{s,t}) \le \frac{1}{2} \sqrt[s]{t-1} \cdot  n^{2-1/s} + \frac{s-1}{2}n.$$
\end{thm}

\begin{cor}\label{cor:2acyclic}
Let $t\ge 3$. Then
    $$a(d,K_{3,t}) = \bigO{t^{1/4} \cdot d^{5/4}}.$$
\end{cor}
\begin{proof}
    Observe that $K_{3,t} \subseteq \hammock{K_{2,t}}$. By the K\H{o}v\'{a}ri-S\'{o}s-Tur\'{a}n Theorem, we obtain $\ex(n,K_{2,t}) \le \frac{1}{2} \sqrt{t-1} \cdot n^{3/2} + \frac{n}{2}$. Applying Theorem~\ref{thm:upperbound-general} together with Observation~\ref{obs:upperbound-subgraph} yields the desired upper bound.
\end{proof}

We observe that there is no hope to derive an upper bound better than the general $\bigO{d^{4/3}}$ one for $a(d,F)$ using a similar application of Theorem~\ref{thm:upperbound-general} when $F$ is a larger complete bipartite graph. Indeed, even when $F\coloneqq K^-_{4,4}$ is $K_{4,4}$ minus an edge, one has $F = \hammock{K_{3,3}}$, and it has been shown in~\cite{Fur96} that $\ex(n,K_{3,3}) \underset{n\to\infty}{\sim} \frac{1}{2} n^{5/3}$.

To cover a larger family of obstructions, we may rely on a generalisation of the K\H{o}v\'{a}ri-S\'{o}s-Tur\'{a}n Theorem to $k$-subdivisions of complete bipartite graphs due to Janzer~\cite{janzer2020extremal}.

\begin{thm}[Janzer, 2020]\label{thm:turan-subdivision}
    Let $s$ and $k$ be fixed positive integers, and $t\ge s$. Then $$\ex(n,K_{s,t}^{1/k}) = \bigO[s,t,k]{n^{1+ \frac{s-1}{sk}}}.$$
\end{thm}

\begin{cor}\label{cor:subdivision-upperbound}
    Let $s$ and $k$ be fixed positive integers, and $t \ge s$. Then
    $$a(d,\hammock{K_{s,t}^{1/k}}) = \bigO[s,t,k]{d^{1 + \frac{s-1}{2sk}}}.$$
\end{cor}

In the next section, we shall derive lower bounds, which indicate that Corollaries~\ref{cor:hammock-tree},~\ref{cor:2acyclic} and~\ref{cor:subdivision-upperbound} cannot be significantly improved.

\section{Lower bounds and related observations}
\label{sec:lowerbounds}
Let us first argue that not all obstructions are worth considering if we seek non-trivial bounds on $a(d,F)$. The following theorem states that if a graph $F$ has several connected components, only one of them sensibly affects the value of $a(d,F)$. Therefore, we may (and will) restrict ourselves to analysing connected obstructions $F$.

\begin{thm}
Suppose $F$ has connected components  $(F_i)_{i \in [r]}$. Then, for every $d \ge 0$, 
 \[\max_{i \in [r]}\; a(d,F_i) \; \le \; a(d,F) \; \le \; \max_{i \in [r]}\; a(d,F_i) + |V(F)|.\]
\end{thm}
\begin{proof}
First, we prove the lower bound.
Let $d \ge 0$. Let $j \in [r]$ be such that $a(d,F_j) = \max\limits_{i \in [r]}\; a(d,F_i)$. Let $G$ be an $F_j$-free graph with $\Delta(G) \le d$ such that $\chi_a(G) = a(d,F_j)$. Since $G$ must also be $F$-free, we have $\chi_a(G) \le a(d,F)$.

Now, we focus on the upper bound.
We proceed by induction on $r$, the number of connected components of $F$.\\
The base case $r=1$ is trivial. Suppose now that $r \ge 2$. Let $d \ge 0$. Let $G$ be an $F$-free graph with $\Delta(G) \le d$ such that $\chi_a(G) = a(d, F)$. Consider whether $G$ is $F_{r-1}$-free or not:
\begin{itemize}
	\item If $G$ is $F_{r-1}$-free, then $\chi_a(G) \le a(d,F_{r-1})$, which clearly proves the upper bound.
	\item Otherwise, $G$ contains a subgraph $G_0$ isomorphic to $F_{r-1}$, thus $G \setminus G_0$ must be $(F \setminus F_{r-1})$-free, therefore, $\chi_a(G \setminus G_0) \le a(d,F \setminus F_{r-1})$. Since $F \setminus F_{r-1}$ has $r-1$ connected components, namely $F_0,\dots,F_{r-2}$, we apply the induction hypothesis and  obtain 
\equa{a(d,F \setminus F_{r-1}) \le \max_{i\in[r-1]} a(d,F_i) + \card{V(F \setminus F_{r-1})}.}
Given an acyclic $k$-colouring of $G \setminus G_0$, we can construct an acyclic $(k+\card{V(G_0)})$-colouring of $G$ by assigning a new unique colour to each vertex of $G_0$. Since $\card{V(G_0)} = \card{V(F_{r-1})}$, and $\card{V(F \setminus F_{r-1})} + \card{V(F_{r-1})} = \card{V(F)}$, we have
\equa{\chi_a(G) \le a(d,F \setminus F_{r-1}) + \card{V( F_{r-1})} \le \max_{i\in[r-1] } a(d,F_i) + \card{V(F)}.}
Clearly, $\max\limits_{i\in[r-1]} a(d,F_i) \le \max\limits_{i\in[r]}\; a(d,F_i)$, therefore the upper bound is proven.
\qedhere
\end{itemize}
\end{proof}

From now on, we will only consider connected obstructions. We now observe that, for every obstruction $F$ that contains a cycle, $a(d,F)$ must be at least linear in~$d$.
\begin{prop}\label{prop:linearlowerbound:cycle}
Let $d\ge 2$ and $F$ be a graph containing a cycle. Then
\[a(d,F) > \frac{d}{2} + 1.\]
\end{prop}
\begin{proof}
For $d \ge 2$ and $g\ge 3$, the existence of $d$-regular graphs of girth $g$ is well known (see~\cite{sachs1963regular} for example). By taking $g$ large enough, we obtain a graph $G$ that is $F$-free and $d$-regular, and therefore satisfies $\chi_a(G) > \frac{d}{2} + 1$, as established by Fertin \emph{et al.} in~\cite{fertin2003acyclic}.
\end{proof}

In the rest of this section, we determine superlinear lower bounds on $a(d,F)$ for various graphs $F$. Alon \emph{et al.}~\cite{AMR91} proved that $a(d,\emptyset) = \Omega(d^{4/3}/(\log d)^{1/3})$ using random graphs: to achieve this, they consider the probability $p \coloneqq 3(\log(n)/n)^{1/4}$, show that $G \gets G(n,p)$ satisfies $\chi_a(G) > \frac{n}{2}$ with high probability, and conclude using the estimation $\Delta(G) \approx np$. We adapt their proof in order to analyse the \emph{random bipartite graph distribution} $G(n,n,p)$ for all $p = p(n)$ such that $9(\log(n)/n)^{1/2} \le p \le 3(\log(n)/n)^{\frac{1}{4}}$. A random bipartite graph $G \gets G(n,n,p)$ is a graph with vertex sets $A$ and $B$, each of size $n$, and where each edge $ab \in A\times B$ is present with probability $p$ (independently of other edges). Our analysis of $G(n,n,p)$ below can easily be adapted to work for $G(n,p)$.

\begin{thm}\label{thm:randomgraph-lowerbound}
     Let $n$ be a (large) integer, and let $t$ be an integer that satisfies $1 \le t \le \frac{1}{9}\sqrt{\frac{n}{\log n}}$.
     Set $p\coloneqq \pt$ and let $G \gets G(n,n,p)$. Then, with high probability, $$\chi_a(G) > \frac{n}{2t}.$$
\end{thm}
\begin{proof}
Let $A$ and $B$ be the vertex sets of $G$. Let $r \le \frac{n}{2t}$ be a positive integer, and fix a partition $(V_1, \ldots, V_r)$ of $V(G)$ into $r$ parts. We derive an upper bound on the probability that this partition is an acyclic colouring of $G$. For $i \in [r]$ let $A_i \coloneqq A \cap V_i$ and $B_i \coloneqq B \cap V_i$. By omitting at most $t$ vertices from each $A_i$, we obtain sub-parts $A'_{1},\ldots,A'_{r}$ of cardinality divisible by $t+1$; doing so removes at most $rt \le \frac{n}{2}$ vertices from $A$. We split every non-empty set $A'_{i}$ into subsets of size $t+1$, and we obtain $k_A\ge \frac{n}{2(t+1)}$ subsets which we label $A''_1, \ldots ,A''_{k_A}$. Similarly, we construct $k_B \ge \frac{n}{2(t+1)}$ subsets $B''_1, \ldots ,B''_{k_B}$ of $B$, each of size $t+1$.

For $(i,j) \in [k_A]\times [k_B]$, let $E_{i,j}$ be the (bad) event that $G[A''_i, B''_j]$ contains a cycle.
If the fixed partition $(V_1,\ldots,V_r)$ corresponds to an acyclic colouring, then in particular, no event $E_{i,j}$ occurs.
Let us first show that the probability of each (bad) event $E_{i,j}$ is not too small.

\begin{claim}\label{claim:proba-acyclic}
    For all $(i,j) \in [k_A]\times [k_B]$, $\pr{E_{i,j}} \ge t^4p^4/8$.
\end{claim}
\begin{subproof}
    Let $H \coloneqq G[A''_i, B''_j]$. We shall bound $\pr{E_{i,j}}$ by the probability that $H$ contains a $4$-cycle.
    For every potential $4$-cycle $C$ in $H$, let $\bX_C$ be the indicator random variable for its presence. Let $\bX \coloneqq \sum_C \bX_C$ count the number of $4$-cycles in $H$.
    There are $\binom{t+1}{2}^2$ possible $4$-cycles in $H$, each appearing with probability $p^4$, therefore 
    \begin{equation}\label{eq:expected-cycles}
        \esp{\bX} = \binom{t+1}{2}^2 p^4. 
    \end{equation}
    
    Let us now evaluate $\esp{\bX^2}$. We do so by grouping the pairs of potential $4$-cycles $(C,C')$ according to their number of common edges. Note that this number is either $0$, $1$, $2$, or $4$.
    
    \begin{enumerate}[label=(\roman*)]
        \item If $C$ and $C'$ share $4$ edges, then $C=C'$. There are $\binom{t+1}{2}^2$ such pairs, and for all of them we have $\esp{\bX_C\bX_{C'}} = p^4$.
        \item If $C$ and $C'$ share exactly $2$ edges, then these two edges are incident since $H$ is bipartite. 
        Given $C$, there are four choices for these two common edges with $C'$, and then $t-1$ choices for the last vertex of $C'$, so the number of such pairs is $4(t-1)\binom{t+1}{2}^2$. Since $|E(C)\cup E(C')|=6$, we have $\esp{\bX_C\bX_{C'}} = p^6$.
        \item If $C$ and $C'$ share exactly one edge, then given $C$ there are four choices for the common edge with $C'$, and $(t-1)^2$ choices for the last two vertices of $C'$, thus we have a total of $4(t-1)^2\binom{t+1}{2}^2$ such pairs. Since $|E(C)\cup E(C')|=7$, we have $\esp{\bX_C\bX_{C'}} = p^7$.
        \item The number of pairs where $C$ and $C'$ are edge-disjoint is less than $\binom{t+1}{2}^4$. In that case we have $|E(C)\cup E(C')|=8$, and so $\esp{\bX_C \bX_{C'}} = p^8$.
    \end{enumerate}
    By linearity of expectation, we obtain
    \begin{align*}
    \esp{\bX^2} & = \sum_{C,C'} \esp{\bX_C\bX_{C'}} \\
    &\le \binom{t+1}{2}^2 p^4 \; + \; 4(t-1)\binom{t+1}{2}^2 p^6 \; + \; 4 (t-1)^2\binom{t+1}{2}^2 p^7 \; + \; \binom{t+1}{2}^4 p^8\\
    & \le \binom{t+1}{2}^2 p^4 \pth{1 + 4(t-1)p^2 + 4(t-1)^2p^3 + \frac{(t+1)^4}{4}p^4}.
\end{align*}
    We have $t \le \frac{1}{9}\sqrt{\frac{n}{\log n}}$, therefore $(t+1)p \le \frac{t+1}{\sqrt{t}}\cdot 3\pth{\frac{n}{\log n}}^{1/4} = 1+o(1)$, and thus
    \begin{equation}\label{eq:expected-cycles-squared}
        \esp{\bX^2} \le 2 \, \binom{t+1}{2}^2 p^4,
    \end{equation}
    for $n$ sufficiently large. We now apply the second moment method together with \eqref{eq:expected-cycles} and \eqref{eq:expected-cycles-squared}, and obtain
    \[\pr{E_{i,j}} \ge \pr{\bX>0} \;\ge\; \frac{\esp{\bX}^2}{\esp{\bX^2}} \ge t^4p^4/8.\qedhere\]
\end{subproof}

We now show that the probability that the fixed partition $(V_1,\ldots,V_r)$ corresponds to an acyclic colouring of $G$  is sufficiently small.
First observe that the events $E_{i,j}$ are mutually independent since they are determined by disjoint subsets of edges of $G$. Therefore, the probability that none of them occurs is 
\begin{align*}
    \pr{\bigwedge_{(i,j)} \overline{E_{i,j}}} &= \prod_{(i,j)} \pth{1-\pr{E_{i,j}}} \\
    &\le \pth{1-t^4p^4/8}^{k_A k_B} & \mbox{by Claim~\ref{claim:proba-acyclic};}\\
    &\le \pth{1-t^4p^4/8}^{\frac{n^2}{4(t+1)^2}} & \mbox{since $k_A,k_B \ge \frac{n}{2(t+1)};$}\\
    &\le \exp\pth{- \frac{t^2 p^4}{32} n^2} & \mbox{using the inequality $1+x \le e^x$}.
\end{align*}
\[  \]
The number of partitions of $V(G)$ into $r \le \frac{n}{2t}$ parts is less than $r^{2n}\le n^{2n}$. 
Using a union bound, we conclude that the probability that there exists an acyclic colouring of $G$ using at most $r$ colours is bounded by
\begin{align*}
\exp\pth{2n\log(n) - \frac{t^2p^4}{32}n^2} &= \exp\pth{2n\log(n) - \frac{81 t^2 \log n}{32\, n t^2}n^{2}} = \exp\pth{-\bigTheta{n\log n}},
\end{align*}
and therefore $\chi_a(G) > \frac{n}{2t}$ with high probability.
\end{proof}

By reformulating the statement of Theorem~\ref{thm:randomgraph-lowerbound}, we immediately obtain the following.
\begin{cor}\label{cor:reformulation}
    Let $n$ be a (large) integer, and let $p \in (0,1)$ such that $9\sqrt{\frac{\log n}{n}} \le p \le 3\pth{\frac{\log n}{n}}^{1/4}$. Sample $G \gets G(n,n,p)$; then, with high probability,
    $$\chi_a(G) = \bigOmega{ \frac{n^{3/2}p^2}{\sqrt{\log n}}}.$$
\end{cor}

We will need some well-known results concerning random graphs to establish lower bounds on $a(d,F)$ for various obstructions $F$.  We include the proofs for completeness since we are considering the bipartite random graph distribution, which is less commonly used.

\begin{lemma}\label{lem:randomgraph-degrees}
    Let $n$ be a large integer, and $\frac{9\log n}{n} \le p < 1$. If $G \gets G(n,n,p)$, then w.h.p.
    \[  \frac{np}{2} \le \deg(v) \le 2np,\]
    for every vertex $v\in V(G)$. In particular, the minimum, average, and maximum degrees of $G$ are all $\Theta(np)$.
\end{lemma}
\begin{proof}
    The degree of a given vertex $u \in V(G)$ follows the binomial distribution $Bin(n,p)$, therefore it has expectation $\mu = np \ge 9\log n$. By the multiplicative Chernoff bounds, we have
    \begin{align*}
        \pr{\deg(u) > 2\mu} < \exp(-\mu/3) &= o\pth{\frac{1}{n}}, \mbox{ and}\\
        \pr{\deg(u) < \mu/2} < \exp(-\mu/8) &= o\pth{\frac{1}{n}}.
    \end{align*}
    We now apply a union bound on $V(G)$ to infer that the probability that there exists a vertex $u\in V(G)$ such that $\deg(u) \notin [\mu/2,2\mu]$ is less than $2n \big(\exp(-\mu/3)+\exp(-\mu/8)\big) =o(1)$, as desired.
\end{proof}

\begin{lemma}\label{lem:randomgraph-subgraph}
    Let $F$ be a connected bipartite graph. Then, for every integer $n\ge 1$ and real $0<p<1$, a random graph $G$ drawn from $G(n,n,p)$ contains less than $4n^{|V(F)|}p^{|E(F)|}$ copies of $F$ with probability at least $1/2$.
\end{lemma}
\begin{proof}
    Suppose $F$ has parts of size $a$ and $b$. Let $\bX$ count the number of (labelled) copies of $H$ in $G$. Define $A^k_n \coloneqq n(n-1)\dots(n-k+1)$. The number of potential labelled copies of $F$ in $G$ is equal to $2A^a_nA^b_n \le 2n^{a+b} = 2n^{|V(H)|}$. The probability of one such labelled copy to exist is equal to $p^{|E(H)|}$. Therefore, $\esp{\bX} \le 2n^{|V(H)|} \; p^{|E(H)|}$, and by Markov's inequality, $\pr{\bX \ge 2\esp{X}} \le 1/2$.
\end{proof}

We are now able to derive a superlinear lower bound on $a(d,F)$ provided that $F$ has average degree at least $4$.

\begin{thm}\label{thm:lowerbound-fixed}
    Let $F$ be a graph of average degree $D = \frac{2|E(F)|}{|V(F)|}$.

    If $4 < D \le 8$, then
    \[a(d,F) = \softOmega{d^{\frac{3}{2} - \frac{1}{D - 2}}}.\]

    If $D \ge 8$, or if $F$ is not bipartite, then \[a(d,F) = \softOmega{d^{4/3}}.\]
\end{thm}
\begin{proof}
    Let us first consider the case where $F$ is not bipartite. By Corollary~\ref{cor:reformulation} and Lemma~\ref{lem:randomgraph-degrees}, fixing $p\coloneqq \pth{\frac{\log n}{n}}^{1/4}$ and letting $G \gets G(n,n,p)$, w.h.p. $\chi_a(G) = \bigOmega{\Delta(G)^{4/3}/\log(\Delta(G))^{1/3}}$. Furthermore, $G$ is bipartite and thus $F$-free, so we conclude that $a(d,F) = \softOmega{d^{4/3}}$.
    
    If $D > 8$, we redefine $D \coloneqq 8$. Let $n$ be a (large) integer, set $p\coloneqq n^{-2/D} \le n^{-|V(F)|/|E(F)|}$, and sample $G\gets G(n,n,p)$. Let $\Delta$ be the maximum degree of $G$; by Lemma~\ref{lem:randomgraph-degrees}, w.h.p. $\Delta = \bigTheta{np} = \bigTheta{n^{1-2/D}}$ . By Corollary~\ref{cor:reformulation}, w.h.p. we have $\chi_a(G) = \bigOmega{n^{3/2}p^{2}/\sqrt{\log n}} = \softOmega{\Delta^{\frac{3}{2} - \frac{1}{D - 2}}}$.

    By Lemma~\ref{lem:randomgraph-subgraph}, we have fewer than $4n^{|V(H)|}p^{|E(H)|} \le 4$ copies of $F$ in $G$ with constant probability. We can remove these copies by deleting one vertex per copy: doing so decreases $\Delta(G)$ and $\chi_a(G)$ by at most $4$.
\end{proof}

Some examples of bipartite graphs with average degree $D \ge 8$ include $K_{5,20}, K_{6,12}, K_{7,10}$, and $K_{8,8}$. We now explicitly state lower bounds on $a(d,F)$ for some notable obstructions $F$ with average degree $4 < D < 8$.

\begin{cor}\label{cor:K_3t-free}
    Let $t\ge 7$. Then
    $$a(d,K_{3,t}) = \widetilde{\Omega}\pth{d^{\frac{5}{4} - \frac{9}{8t - 12}}}$$
\end{cor}
\begin{proof}
    $K_{3,t}$ has average degree $D = \frac{6t}{3+t} > 4$ (assuming $t\ge 7$). We have $\frac{3}{2} - \frac{1}{D-2} = \frac{5}{4} - \frac{9}{8t - 12}$, and thus by Theorem~\ref{thm:lowerbound-fixed}, $a(d,K_{3,t}) = \softOmega{d^{\frac{5}{4} - \frac{9}{8t - 12}}}$.
\end{proof}

\begin{cor}\label{cor:K_4t-free}
    Let $t\ge 5$. Then
    $$a(d,K_{4,t}) = \widetilde{\Omega}\pth{d^{\frac{4}{3} - \frac{8}{9t - 12}}}$$
\end{cor}
\begin{proof}
    $K_{4,t}$ has average degree $D = \frac{8t}{4+t} > 4$ (assuming $t\ge 5$). We have $\frac{3}{2} - \frac{1}{D-2} = \frac{4}{3} - \frac{8}{9t - 12}$, and thus by Theorem~\ref{thm:lowerbound-fixed}, $a(d,K_{4,t}) = \softOmega{d^{\frac{4}{3} - \frac{8}{9t - 12}}}$.
\end{proof}

\begin{cor}\label{cor:subdivision-lowerbound}
    Let $s \ge 2$ and $k \ge 2$ be (fixed)  positive integers. Let $t$ be sufficiently large. Then
    $$a(d,\hammock{K_{s,t}^{1/k}}) = \bigOmega{d^{1 + \frac{s-1}{2sk} - \frac{1}{t}}}$$
\end{cor}
\begin{proof}
    $K_{s,t}^{1/k}$ has $st(k-1)+s+t$ vertices and $stk$ edges. The two anchor vertices of $\hammock{K_{s,t}^{1/k}}$ add $2$ vertices and $st(k-1) + s+t$ edges. 
    By Theorem~\ref{thm:lowerbound-fixed}, we have $a(d,\hammock{K_{s,t}^{1/k}}) = \softOmega{d^{\frac{3}{2} - \frac{1}{D-2}}}$, where $D$ is the average degree of $\hammock{K_{s,t}^{1/k}}$. One can check by computation that 
    \begin{align*}
    \frac{1}{D-2}  &= \frac{st(k-1) + s + t + 2}{2stk - 4} \\
    &= \frac{st(k-1) +s + t + 2}{2stk}\cdot \frac{1}{1-2/stk} \\
    &< \frac{st(k-1) +s + t + 2}{2stk}\cdot\pth{1+ \frac{4}{stk}} & \mbox{since $\forall x\in (0,1/2)$, } \frac{1}{1-x} < 1+2x;\\
    &\le \frac{1}{2} - \frac{1}{2k} + \frac{1}{2sk} + \frac{1}{2st} + \frac{1}{stk} + \frac{4stk}{2(stk)^2} & \mbox{since } s+t+2 \le st;\\
    &\le \frac{1}{2} - \frac{s-1}{2sk} + \frac{1}{t} & \mbox{since } s,k \ge 2. 
    \end{align*}
    We conclude that 
    $a(d,\hammock{K_{s,t}^{1/k}})  = \bigOmega{d^{1 + \frac{s-1}{2sk} - \frac{1}{t}}}$.
\end{proof}


\section{Acyclic obstructions} \label{sec:forests}

In this section, we consider obstructions $F$ that are acyclic; those are outside the scope of Proposition~\ref{prop:linearlowerbound:cycle}, meaning that there is \emph{a priori} no linear lower bound on $a(d,F)$. Indeed, we show that $a(d,F)$ is sublinear in $d$ in that case. We denote by $G^{1/2}$ the \emph{subdivision} of a graph $G$. We call a \emph{subdivided tree} a tree obtained by taking a subgraph of a tree subdivision, or equivalently a tree in which every pair of vertices of degree at least $3$ are at even distance from each other. A \emph{non-subdivided tree} is a tree that is not a subdivided tree.

Choi \emph{et al.}~\cite{choi2019characterization}, characterised the obstructions $F$ for which the acyclic chromatic number of $F$-free graphs is bounded.

\begin{thm}[Choi, Kim, Park, 2019]\label{thm:choi}
   Let $F$ be a connected graph. The class of $F$-free graphs has a bounded acyclic chromatic number if and only if $F$ is a subdivided tree.  
\end{thm}

We show that Theorem~\ref{thm:choi} can also be derived from an earlier characterisation of classes of graphs with bounded acyclic chromatic number due to 
Dvo\v{r}\'{a}k~\cite[Corollary~4]{dvovrak2008forbidden}.

\begin{thm}[Dvořák, 2008] \label{thm:chia-constant}
Let $\mc{G}$ be a class of graphs of bounded chromatic number.
$\mc{G}$ has bounded acyclic chromatic number if and only if there exists a constant $c$ such that for any graph $H$, if $H^{1/2}$ is a subgraph of a graph in $\mc{G}$ then $\chi(H) \le c$.
\end{thm}


\begin{proof}[Proof of Theorem~\ref{thm:choi}]
Let $F$ be a connected graph and $\cG$ be the class of $F$-free graphs. We show that $\cG$ has a bounded acyclic chromatic number if and only if $F$ is a subdivided tree. We may assume that $F$ is a tree. Otherwise, we have $a(d,F) > d/2 + 1$ by Proposition~\ref{prop:linearlowerbound:cycle}. By Lemma \ref{lem:tree-min-degree}, graphs in $\cG$ are $(|V(F)|-2)$-degenerate, and therefore have bounded chromatic number.

Suppose $F$ is a subdivided tree: take a tree $T$ such that $F\subseteq T^{1/2}$, and observe that $\cG$ is also $T^{1/2}$-free. Let $H$ be a graph and suppose there exists $G \in \mc{G}$ such that $H^{1/2} \subseteq G$. Then $H$ is $T$-free, otherwise $T^{1/2} \subseteq H^{1/2} \subseteq G$, a contradiction. By Lemma~\ref{lem:tree-min-degree}, $H$ is $(|V(T)|-2)$-degenerate, and therefore $\chi(H)\le |V(T)|-1$. By Theorem \ref{thm:chia-constant}, $\cG$ has bounded acyclic chromatic number.

Now, suppose $F$ is a non-subdivided tree. Consider the complete graph $K_n$ where $n \rightarrow \infty$; it is clear that $K_n^{1/2}$ is $F$-free, thus $K_n^{1/2} \in \cG$, yet $\chi(K_n)$ is not bounded by any constant $c$. By Theorem \ref{thm:chia-constant}, $\cG$ has unbounded acyclic chromatic number.
\end{proof}

We now look at acyclic obstructions that are not subdivided trees.

In 2005, Wood~\cite{wood2005acyclic} determined bounds on the acyclic chromatic number of graph subdivisions; he showed that $\sqrt{n/2} < \chi_a(K^{1/2}_n) < \sqrt{n/2} + \frac{5}{2}$. 
Observe that $K^{1/2}_n$ is $T$-free for every non-subdivided tree $T$. Moreover, $K^{1/2}_n$ is $C_4$-free and $\Delta(K^{1/2}_n) = n-1$ for all $n\ge 3$, hence by definition we have $a(d,\{C_4,T\}) \ge \chi_a(K^{1/2}_{d+1})$. This lets us derive the following proposition.

\begin{prop}\label{prop:not-subdivided}
If $T$ is a non-subdivided tree, then\[a(d,\{C_4,T\}) > \sqrt{\frac{d+1}{2}}.\]
\end{prop}

In the rest of this section, we first show that the lower bound given by Proprosition~\ref{prop:not-subdivided} is tight by proving that
for every fixed tree $T$, $a(d,\{C_4,T\}) = \bigO[T]{\sqrt{d}}$. We then consider the case where $T$ is a single obstruction and obtain the upper bound $a(d,T) = \bigO[T]{d^{2/3}}$. While we were not able to derive a matching lower bound in that case, we suspect that this upper bound might be tight, at least up to a polylogarithmic factor.

In order to derive these results, we consider the more general family of $t$-degenerate graphs for some fixed integer $t$. Given such a graph $G$, it has an acyclic orientation $\vec{G}$ of maximum out-degree at most $t$.
In $\vec{G}$, we will focus on the \emph{antidirected paths/cycles}, which are paths/cycles of $G$ that contain no directed subpath of length $2$ in $\vec{G}$. Note that an antidirected cycle is necessarily even.

\begin{thm}\label{thm:degenerate+C4free}
Let $G$ be a $C_4$-free $t$-degenerate graph of maximum degree $\Delta$. Then
\[ \chi_a(G) = \bigO{\sqrt{t^5\Delta}}, \quad \mbox{as $\Delta \to \infty$.}\]
\end{thm}

\begin{proof}
Since $G$ is $t$-degenerate, we can find an acyclic orientation $\vec{G}$ of $G$ where each vertex has out-degree at most $t$.
Let $\Gamma_0$ be the set of pairs of vertices linked by a directed path of length at most $2$ in $\vec{G}$, and let $\Pi$ be the set of antidirected cycles, which are exactly the $\Gamma_0$-free cycles.

The graph $(V(G),\Gamma_0)$ is $(t^2+t)$-degenerate, so we can greedily construct a $\Gamma_0$-proper $(t^2+t+1)$-colouring $\phi_0$ of $G$.
Now, we will construct a (possibly improper) $\Pi$-acyclic $\bigO{\sqrt{t\Delta}}$-colouring $\phi_1$ of $G$. 
The solution will follow by taking the Cartesian product of $\phi_0$ and $\phi_1$, which by construction is a $\Gamma_0$-proper $\Pi$-acyclic $\bigO{\sqrt{t^5\Delta}}$-colouring of $G$, and is therefore a proper acyclic $\bigO{\sqrt{t^5\Delta}}$-colouring of $G$ by Proposition~\ref{prop:gamma-free-cycles}.

\begin{claim}
\label{claim:pi}
$\degmax{\Pi} \le (t\Delta)^{\ell-1}$, for every integer $\ell\ge 3$
\end{claim}

\begin{subproof}
Let $v_0\in V(G)$.
The number of antidirected paths of length $2\ell-2$ that begin in $v_0$ is at most $2(t\Delta)^{\ell-1}$ (because every two steps in that path follow an outgoing edge and an ingoing edge, for a total of at most $t\Delta$ choices, and we have the choice to begin the path either with an outgoing or an ingoing edge, that doubles the number of choices). Since $G$ is $C_4$-free, there is at most $1$ way to close such a path into a $2\ell$-cycle. Doing that, we count each antidirected $2\ell$-cycle containing $v_0$ twice, so their number is at most $(t\Delta)^{\ell-1}$, as desired.
\end{subproof}

We may now apply Theorem~\ref{thm:main} on $G$, with an empty set of constraints and the set of cycles $\Pi$, after fixing $\tau \coloneqq \Phi \sqrt{t\Delta}$, where $\Phi\coloneqq \frac{1+\sqrt{5}}{2}$ satisfies $\Phi^2=1+\Phi$. We obtain that there are at least $\tau^{|V(G)|}$ (possibly improper) $\Pi$-acyclic $\ceil{K}$-colourings of $G$, where 
\begin{align*}
    K &\coloneqq \tau + \sum_{\ell\ge 3} \frac{\degmax{\Pi})}{\tau^{2\ell-3}} = \tau \pth{1+\sum_{\ell\ge 3} \frac{\degmax{\Pi})}{\tau^{2\ell-2}}}  \\
    &\le \Phi \sqrt{t\Delta}\pth{1 + \sum_{\ell \ge 3} \frac{(t\Delta)^{\ell-1}}{(\Phi^2 t\Delta)^{\ell-1}}} & \mbox{by Claim~\ref{claim:pi}}\\
    &\le \Phi\sqrt{t\Delta} \pth{1+\frac{1}{\Phi^2(\Phi^2-1)}} = \sqrt{t\Delta} \pth{\Phi + \frac{1}{\Phi^2}} = 2\sqrt{t\Delta}.
    && \qedhere
\end{align*}
\end{proof}

We note that, in the above proof, we could replace the assumption that $G$ is $C_4$-free with the weaker assumption that $\vec{G}$ has no antidirected cycle of length $4$. 

Let $T$ be a tree on $t$ vertices. If a graph $G$ is not $(t-2)$-degenerate, then it contains a subgraph $H$ of minimum degree at least $t-1$, and thus by Lemma~\ref{lem:tree-min-degree}, $H$ contains $T$ as a subgraph. 
We conclude that any $T$-free graph $G$ is $(t-2)$-degenerate. Hence, we have the following result as a corollary.

\begin{cor}
\label{cor:C4-F-free}
For every tree $T$ on $t$ vertices, 
\[a(d,\{C_4,T\}) = \bigO{\sqrt{t^5d}}, \quad \mbox{as $d \to \infty$.}\]
\end{cor}

\medskip
We now consider the case where $G$ may contain $4$-cycles. The setup of the proof is similar to before, but we need an extra step that treats pairs of vertices of large codegree separately.

\begin{thm}\label{thm:degenerate}
Let $G$ be a $t$-degenerate graph of maximum degree $\Delta$. Then
\[ \chi_a(G) = \bigO{t^{8/3} \Delta^{2/3}}, \quad \mbox{as $\Delta \to \infty$.}\]

\end{thm}

\begin{proof}
Since $G$ is $t$-degenerate, we can find an acyclic orientation $\vec{G}$ of $G$ where each vertex has out-degree at most $t$. 
Let $\Gamma_0$ be the set of pairs of vertices linked by a directed path of length at most $2$ in $\vec{G}$, and let $\Pi$ be the set of antidirected cycles.

The graph $(V(G),\Gamma_0)$ is $(t^2+t)$-degenerate, so we can greedily construct a $\Gamma_0$-proper $(t^2+t+1)$-colouring $\phi_0$ of $G$.
Now, we  will construct a (possibly improper) $\Pi$-acyclic $\bigO{(t\Delta)^{2/3}}$-colouring $\phi_1$ of $G$. 
Observe that if there exists such a colouring, the Cartesian product of $\phi_0$ and $\phi_1$ yields a $\Gamma_0$-proper $\Pi$-acyclic
$\bigO{t^{8/3}\Delta^{2/3}}$-colouring of $G$, which by Proposition~\ref{prop:gamma-free-cycles} is a proper acyclic
$\bigO{t^{8/3}\Delta^{2/3}}$-colouring of $G$.

Let $\Gamma_1$ consist of all pairs of vertices $\{u,v\}$ such that $|N^-_{\vec{G}}(u)\cap N^-_{\vec{G}}(v)| \ge (t\Delta)^{1/3}$.
\begin{claim}
\label{claim:degree}
The maximum constraint-degree is $\Delta(\Gamma_1) \le (t\Delta)^{2/3}$.
\end{claim} 

\begin{subproof}
Let $v_0 \in V(G)$, and let $A_0$ be the set of out-going arcs from $N^-(v_0)$.

On the one hand, each vertex $w\in N^-(v_0)$ is incident to at most $t$ arcs from $A_0$ by assumption on the maximum out-degree of $\vec{G}$. Thus   $t\card{N^-(v_0)} \ge |A_0|$. On the other hand, each vertex $u \in N_{\Gamma_1}(v_0)$ is incident to at least
$(t\Delta)^{1/3}$ arcs from $A_0$ by definition of $\Gamma_1$. Hence we have $t\card{N^-(v_0)} \ge |A_0| \ge (t\Delta)^{1/3}\card{N_{\Gamma_1}(v_0)}$.

Since $\Delta \ge \card{N^-(v_0)}$ and $\deg_{\Gamma_1}(v_0)=\card{N_{\Gamma_1}(v_0)}$, we have
$\deg_{\Gamma_1}(v_0) \le (t\Delta)^{2/3}$.
\end{subproof} 
We now let $\Pi_1 \subseteq \Pi$ be the set of $\Gamma_1$-free antidirected cycles in $\vec{G}$. 

\begin{claim}
\label{claim:cycles}
$\degmax{\Pi_1} \le \frac{1}{2}(t\Delta)^{\ell-2/3}$, for every integer $\ell\ge 2$
\end{claim}

\begin{subproof}
Let $v_0\in V(G)$. 
On one hand, the number of antidirected paths of length $2\ell-2$ that begin with an in-going arc from $v_0$ is at most $\deg^-(v_0) t^{\ell-1}\Delta^{\ell-2}$, since there are at most $t$ choices for any out-going arc in $\vec{G}$. Given such a path $(v_0, \ldots, v_{2\ell-2})$, if we assume that $\{v_0, v_{2\ell-2}\}\notin \Gamma_1$, there are at most $(t\Delta)^{1/3}$ choices in order to close it into an antidirected cycle of length $2\ell$. So the number of $\Gamma_1$-free antidirected $2\ell$-cycles that contain an in-going edge from $v_0$ is at most $\frac{1}{2}\deg^-(v_0)t^{\ell-2/3}\Delta^{\ell-5/3}$ --- the factor $\frac{1}{2}$ in the bound comes from the fact that each cycle is counted twice in the previous process, depending on the direction in which we construct it.
On the other hand, the number of antidirected paths of length $2\ell-3$ that begin with an out-going arc of $v_0$ is at most $\deg^+(v_0)(t\Delta)^{\ell-2}$. Given such a path $(v_0, \ldots, v_{2\ell-3})$, we may close it into a cycle of length $2\ell$ by first picking a vertex $v_{2\ell-1}\in N^+(v)$, then a vertex $v_{2\ell-2}\in N^-(v_{2\ell-3}, v_{2\ell-1})$. The number of choices is at most $\deg^+(v_0)(t\Delta)^{1/3}$ if we assume that $\{v_{2\ell-3}, v_{2\ell-1}\} \notin \Gamma_1$. The number of $\Gamma_1$-free antidirected $2\ell$-cycles that contain an out-going edge from $v_0$ is therefore at most $\frac{1}{2}\deg^+(v_0)t^{\ell-2/3}\Delta^{\ell-5/3}$.
Overall, the number of $\Gamma_1$-free antidirected $2\ell$-cycles that contain $v_0$ is at most $\frac{1}{2}(\deg^+(v_0)+\deg^-(v_0))t^{\ell-2/3}\Delta^{\ell-5/3} \le \frac{1}{2}(t\Delta)^{\ell-2/3}$.
\end{subproof}

We may now apply Theorem~\ref{thm:main} on the graph $G$, with the set of constraints $\Gamma_1$ and the set of cycles $\Pi_1$, after fixing $\tau \coloneqq \frac{\sqrt{2}}2(t\Delta)^{2/3}$. We obtain that there are at least $\tau^{|V(G)|}$ $\Gamma_1$-proper $\Pi_1$-acyclic $\ceil{K}$-colourings of $G$, where 
\begin{align*}
    K &\coloneqq \Delta(\Gamma_1) + \tau + \sum_{\ell\ge 2} \frac{\degmax{\Pi_1}}{\tau^{2\ell-3}} \\
    &\le \frac{2+\sqrt{2}}2(t\Delta)^{2/3} + \frac{\sqrt{2}}2\sum_{\ell \ge 2} (t\Delta)^{4/3-\ell/3} & \mbox{by Claims~\ref{claim:degree} and~\ref{claim:cycles}}\\
    &\le \frac{2+\sqrt{2}}2(t\Delta)^{2/3} + \frac{\sqrt{2}}2 \cdot \frac{(t\Delta)^{2/3}}{1-(t\Delta)^{-1/3}} =  \pth{1+\sqrt{2}}(t\Delta)^{2/3} + \bigO{(t\Delta)^{1/3}}, \hspace{-100pt}
\end{align*}
as $(t\Delta)\to \infty$.

Finally, by Proposition~\ref{prop:gamma-free-cycles}, a $\Gamma_1$-proper $\Pi_1$-acyclic $\ceil{K}$-colouring of $G$ is in particular a (possibly improper) $\Pi$-acyclic $\ceil{K}$-colouring of $G$, as desired for $\phi_1$.
\end{proof}

Since --- as previously established --- for every tree $T$ on $t$ vertices, every $T$-free graph $G$ is $(t-2)$-degenerate, we have the following result as a corollary.

\begin{cor} \label{cor:forest}
For every tree $T$ on $t$ vertices, 
\[ a(d,T) = \bigO{t^{8/3}d^{2/3}}.\]
\end{cor}

\section{Cycle obstructions}\label{sec:cycle}

Let $t \ge 2$. The even cycle $C_{2t}$ is a bipartite graph such that removing any of its vertices yields an acyclic subgraph, therefore by Corollary~\ref{cor:1acyclic}, $a(d,C_{2t}) = \bigO{\sqrt{t} \cdot d}$. This linear upper bound has a first-order dependency on $t$, the size of the obstruction. This section shows that this dependency is only of second order for cycle obstructions. For technical reasons,  we have to treat $C_4$-free graphs separately.

\subsection{\texorpdfstring{$C_4$}{C₄}-free graphs}

We begin with a bound of the number of cycles that contain a fixed vertex in $C_4$-free graphs.

\begin{lemma}
\label{lem:Ncycles-C4free}
Given a $C_4$-free graph $G$, let $\Omega_G$ denote its set of cycles. Then, for every integer $\ell \ge 3$,
\[ \degmax{\Omega_G} \le \frac{\Delta}{2}(\Delta-1)^{2\ell-3}.\]
\end{lemma}

\begin{proof}
Let $v_0\in V(G)$ be a fixed vertex, and $\ell \ge 3$ a fixed integer. There are at most $\Delta(\Delta-1)^{2\ell-3}$ paths of length $2\ell-2$ starting from $v_0$, which can be computed by performing a breadth-first search of depth $2\ell-2$ from $v_0$. For each of these paths, there is a unique way to close it into a $2\ell$-cycle with a common neighbour of its extremities since $G$ is $C_4$-free (if two vertices have two common neighbours, then they form a $C_4$). Each cycle is counted twice (clockwise and anticlockwise) with this enumeration, hence we divide the total by $2$.
\end{proof}

We note that the bound provided by Lemma~\ref{lem:Ncycles-C4free} is asymptotically tight since in the incidence graph $G$ of a projective plane, of maximum degree $\Delta$ (which can be any prime power plus one), it is straightforward to show that $\degmax{\Omega_G} \ge \frac{\Delta}{2}(\Delta-\ell)^{2\ell-3}$.

We may now prove the upper bound on $a(d,C_4)$.

\begin{thm}\label{thm:c4free}
For every $d\ge 2$, \[a(d,C_4) < 2.7627d.\]
\end{thm}

\begin{proof}

Let $G$ be a $C_4$-free graph of maximum degree $\Delta$.
We prove Theorem~\ref{thm:c4free} through an application of Theorem~\ref{thm:main} with $\Gamma=E(G)$ and $\Pi=\Omega_G$. We fix $\alpha \coloneqq \argmin\limits_{x > 1} \pth{x + \frac{1}{2(x^3 - x)}} \approx 1.4576$ and $\tau \coloneqq \alpha(\Delta-1)$, so that there exist $\tau^{V(G)}$ proper acyclic $\ceil{K}$-colourings of $G$, with
\begin{align*}
    K &\coloneqq \Delta + \tau + \sum_{\ell\ge 2} \frac{\degmax{\Omega_G}}{\tau^{2\ell-3}} \\
    & \le \Delta + \alpha(\Delta-1) + \frac{\Delta}{2}\sum_{\ell\ge 3} \frac{1}{\alpha^{2\ell-3}} & \mbox{by Lemma~\ref{lem:Ncycles-C4free}}\\
    & = \Delta \pth{1 + \alpha + \frac{1}{2(\alpha^3-\alpha)}} - \alpha \\
    & < 2.7627 \Delta - 1.4575. &&\qedhere
\end{align*}
\end{proof}

Now, we extend this result to larger cycle obstructions.

\subsection{Larger cycle obstructions}
In this section, we prove an upper bound on $a(d,C_{2t})$ that is linear in $d$ with only a second-order dependency on $t$. 

We will need the following lemma that bounds the number of edges with an extremity of a large degree in a graph of bounded maximum average degree. Given a bipartite graph $H=(X,Y,E)$ with $|X|\le |Y|$ and an integer $d\ge 1$, we say that an edge $e=xy\in X\times Y$ is \emph{$d$-branching} if $\deg(y) \ge d$. 

\begin{lemma}
\label{lem:branching}
Let $d\ge 1$ be an integer, and let $H$ be a bipartite graph of maximum average degree at most $d$. Then the number of $d$-branching edges in $H$ is at most $d\,|X|$.
\end{lemma}

\begin{proof}
Let $Y_d \subseteq Y$ be the set of vertices with degree at least $d$, and let $H' \coloneqq H[Y_d,N(Y_d)]$ be the bipartite subgraph of $H$ induced by the $d$-branching edges.
Since the average degree of $H'$ is at most $d$, we infer that 
\begin{align*}
d &\ge \frac{2|E(H')|}{|Y_d|+|N(Y_d)|} \ge \frac{2|E(H')|}{|E(H')|/d+|X|},
\end{align*}
and hence $|E(H')| \le d\,|X|$, as desired.
\end{proof}

We now state the main result of this section.

\begin{thm} \label{thm:C2t}
Let $t\ge 3$ be a fixed integer, and let $d\ge 8t^3$. Then
\[ a(d,C_{2t}) \le 2d + \bigO{td^{2/3}}.\]
\end{thm}

\begin{proof}

Let us fix an integer $t\ge 3$.
Let $G$ be a $C_{2t}$-free graph of maximum degree $\Delta\ge 8t^3$. Let $v_0\in V(G)$ be an arbitrary vertex, and let $X_i$ be the set of vertices at distance $i$ from $v_0$, for every $i\in \{0, \ldots, t\}$. We denote $H_i \coloneqq G[\bX_{i}, \bX_{i+1}]$ for every $1\le i\le t-1$. 

\begin{claim}
\label{prop:pikhurko}
The maximum average degree of $H_i$ is at most $2t$ for every $1\le i \le t-1$, and that of $G[X_i]$ is at most $2t-3$ for $i=1$ and at most $4t$ for $2\le i \le t-1$.
\end{claim}

\begin{subproof}
The fact that the maximum average degree of $G[X_1]$ is at most $2t-3$ follows from Theorem~\ref{thm:turan-paths}, because $G[X_1]$ is $P_{2t-1}$-free.
The rest is a reformulation of~\cite[Equation~(2)]{Pik12}.
\end{subproof}

We fix a real parameter $\gamma \in (0,1)$ that satisfies $\Delta^\gamma \ge 2t$. A \emph{$\gamma$-special pair} in $G$ is a pair of vertices $(u,v)$ whose codegree is $\deg(u,v) \ge \Delta^\gamma$. We let $\Gamma_0$ be the set of $\gamma$-special pairs in $G$.

\begin{claim}
\label{lem:special}
$\Delta(\Gamma_0) \le 4t\Delta^{1-\gamma}$.
\end{claim}

\begin{subproof} 
Let $E$ be the set of edges between $N_{\Gamma_0}(v_0)$ and $N_G(v_0)$.
We let $E_1 \coloneqq E \cap E(G[X_1])$, and $E_2 \coloneqq E \cap E(H_1)$.
We have $|E_1|\ge \Delta^\gamma |N_{\Gamma_0}(v_0)\cap X_1|/2$ and $|E_2| \ge \Delta^\gamma |N_{\Gamma_0}(v_0) \cap X_2|$.
On the one hand, by Claim~\ref{prop:pikhurko} the average degree of $G[X_1]$ is at most $2t-3$ --- hence $|E_1|\le t\Delta$ --- and the average degree of $H_1$ is at most $2t$. If $|X_2|\le |X_1|$, this directly implies that $|E_2|\le 2t\Delta$.
If on the other hand $|X_2|\ge|X_1|$, then since $\Delta^\gamma \ge 2t$, the edges in $E_2$ are $2t$-branching in $H_1$. By Lemma~\ref{lem:branching}, we infer that $|E_2|\le 2t\Delta$.
We conclude that we have $|N_{\Gamma_0}(v_0)\cap X_1| \le 2|E_1|\Delta^{-\gamma} \le 2t\Delta^{1-\gamma}$, and $|N_{\Gamma_0}(v_0)\cap X_2| \le |E_2|\Delta^{-\gamma} \le 2t\Delta^{1-\gamma}$. The result follows.
\end{subproof}

\begin{claim}
\label{lem:Ncycles}
Let $\Pi$ be the set of $\Gamma_0$-free even cycles in $G$. Then $\degmax{\Pi} \le \bigO{t\Delta^{2\ell-3+\gamma}}$, for every $\ell \ge 3$.
\end{claim}

\begin{subproof}
Let $C=(v_0, v_1, \ldots, v_{2\ell-1})$ be a $\Gamma_0$-free cycle of length $2\ell$.
%
%
By Claim~\ref{prop:pikhurko}, the maximum average degree in $H_1$ and  $H_2$ is at most $2t$, and that in $G[X_2]$ is at most $4t$. Hence we may orient the edges of $G$ so that the maximum out-degree is at most $2t$ in $H_1$ and in $H_2$, and at most $4t$ in $G[X_2]$ --- this is possible because these three subgraphs of $G$ are edge-disjoint.
If $v_2 \to v_3$ in the orientation of $G$, then there are at most $\Delta^2$ choices for the path $v_0,v_1,v_2$, at most $8t$ choices for the out-going arc $v_2 \to v_3$, and finally at most $\Delta^{2\ell-5+\gamma}$ choices for the path $v_3, v_4, \ldots, v_0$.
If on the other hand $v_3 \to v_2$, then there are at most $\Delta^{2\ell-4}$ choices for the path $v_0, v_{2\ell-1}, \ldots, v_3$, at most $8t$ choices for the arc $v_3\to v_2$, and finally at most $\Delta^{\gamma}$ choices for $v_1$.
Overall, the total number of choices for a $\Gamma$-free cycle of length $2\ell$ that contains $v_0$ is at most $16t\Delta^{2\ell-3+\gamma}$.\qedhere
\end{subproof}
\begin{claim}
\label{lem:NC4}
$\Delta_4(\Pi) \le 2t\Delta^{1+\gamma}$
\end{claim}
\begin{subproof}
Let $v_0\in V(G)$, and
let $U$ be the set of vertices $u$ such that $2\le \deg(u,v_0) < \Delta^\gamma$. Every $4$-cycle in $\Pi$ that contains $v_0$ must go through a vertex in $U$.

Let $H \coloneqq G[N(v_0)\cup U]$, and let $p(v)$ denote the number of $\Gamma_0$-free paths of length $2$ in $H$ from any given vertex $v \in N(v_0)$ to another vertex $v'\in N(v_0)$ (this means that $vv'\notin \Gamma_0$). Assume for the sake of contradiction that $p(v) \ge 2t\Delta^\gamma$ for every $v\in N(v_0)$. We claim that it is possible to construct  a path $P=u_0, u_1, \ldots, u_{2t-2}$ greedily such that $u_{2i} \in N(v_0)$ for every $i<t$. Indeed, if we have constructed the $2i$-subpath $P_i$ of $P$ for some integer $0\le i < t$, then we may extend it to a path of length $2i+2$ with one of the $p(u_{2i})$ $\Gamma_0$-free paths of length $2$ starting in $u_i$, which we choose to be disjoint from $V(P_i)$. There are at most $i \Delta^\gamma < t\Delta^\gamma$ of them that go through the set $\{u_{2j+1}\}_{j<i}$ (because there are $i$ choices for $j$, and less than $\Delta^\gamma$ choices for a neighbour of $u_{2j+1}$ in $N(v_0)$ given $j$), and there are at most $i\Delta^\gamma < t\Delta^\gamma$ of them that go through the set $\{u_{2j}\}_{j<i}$ (because given $u_{2j}$ such that $u_{2i}u_{2j}\notin \Gamma_0$, there are at most $\Delta^\gamma$ common neighbours of $u_{2i}$ and $u_{2j}$), so this is always possible. This yields a contradiction since $P+v_0$ forms a $2t$-cycle in $G$.
We conclude that the minimum $p(v)$ is at most $2t\Delta^\gamma$. 
Let $\hat{H}$ be the multigraph of vertex-set $N(v_0)$ where there is one edge between two vertices $x,y\in N(v_0)$ for each $\Gamma_0$-free path of length $2$ in $H$ between $x$ and $y$.
Then by construction $p(v) = \deg_{\hat{H}}(v)$ for every $v\in N(v_0)$, and so by Lemma~\ref{lem:average-min-degree} applied on $\hat{H}$, the average $p(v)$ is at most $4t\Delta^\gamma$. Therefore, we infer that the number of $\Gamma_0$-free paths of length $2$ between pairs of vertices of $N(v_0)$ is at most $2t\Delta^{1+\gamma}$. This is an upper bound on the number of $\Gamma_0$-free $4$-cycles that contain $v_0$, so the result follows. \qedhere
\end{subproof}

We now fix $\gamma \coloneqq 1/3$.
We apply Theorem~\ref{thm:main} with $\Gamma = E(G) \cup \Gamma_0$ and $\Pi$ the set of $\Gamma_0$-free even cycles in $G$. By fixing $\tau \coloneqq \Delta + \Delta^{2/3}$, this yields a $\Gamma$-proper $\Pi$-acyclic $\ceil{K}$-colouring of $G$, with
\begin{align*}
    K & \coloneqq \Delta(\Gamma) + \tau + \sum_{\ell \ge 2} \frac{\degmax{\Pi}}{\tau^{2\ell-3}}\\
    &\le 2\Delta + (4t+1)\Delta^{2/3} + \bigO{t\Delta^{1/3}} \sum_{\ell\ge 2} \pth{\frac{\Delta}{\tau}}^{2\ell-3} & \mbox{by Claims~\ref{lem:special},~\ref{lem:Ncycles},~and~\ref{lem:NC4}};\\
    &\le 2\Delta + (4t+1)\Delta^{2/3} + \bigO{t\Delta^{1/3}}\,\frac{\tau \Delta}{\tau^2-\Delta^2} = 2\Delta + \bigO{t\Delta^{2/3}}.\hspace{-100pt}
\end{align*}

The result follows by noting that, by Proposition~\ref{prop:gamma-free-cycles}, a $\Gamma$-proper $\Pi$-acyclic $\ceil{K}$-colouring of $G$ is in particular a proper acyclic $\ceil{K}$-colouring of $G$.
\end{proof}

\subsection{Girth 7: below the \texorpdfstring{$2\Delta$}{2Δ} threshold}

When $\F$ is a collection of cycles, it seems that, given an $\F$-free graph $G$, there is no better general upper bound on $\Delta_{2\ell}(\Omega_G)$ than $\Delta(G)^{2\ell - C_\F}$ for some constant $C_\F$. In particular, if we wish to apply Theorem~\ref{thm:main} to obtain an upper bound on $a(d,\F)$, we need $\tau > d$. Since we have $\Delta(\Gamma)=d$ in that setting, there is no hope of obtaining an upper bound below $2d$. 
With a more involved technique that uses properties of proper colouring in sparse graphs, we can obtain an upper bound below that threshold for a small family of cycles $\F$.
Our proof will rely on the following Coupon-Collector Lemma from~\cite{HuPi23}.

\begin{lemma}[Hurley, Pirot, 2023]\label{lem:coupon}
    Suppose we have random non-empty lists $\boldL_1, \dots, \boldL_d$, each of which takes values in the finite subsets of $\NN$. Fix some integer $t\ge 1$, and define the random variable $\boldX \coloneqq \#\set{i \in [d] : |\boldL_i| \le t}$. Now choose an element $\mathbf{\sigma}(i)$ of $\boldL_i$ uniformly at random for each $i \in [d]$ and define the random variable $\boldL \coloneqq [k] \setminus \mathbf{\sigma}([d])$, for some integer $k$. Then
    $$\esp{|\boldL|} \ge (k-\esp{\boldX}) \exp\pth{-\frac{t+1}{t}\frac{d}{k-\esp{\boldX}}}.$$
\end{lemma}

We can now state the main result of this section.

\begin{thm}
\label{thm:girth7}
For every graph $G$ of maximum degree $\Delta$ and of girth (at least) $7$, 
\[ \chi_a(G) \le \frac{\Delta}{W(1)} + \bigO{\sqrt{\Delta}} < 1.7633\,\Delta + \bigO{\sqrt{\Delta}}, \quad \mbox{as $\Delta \to \infty$,}\]
where $W(1) \approx 0.5671$ is the omega constant, which uniquely satisfies the equation $\frac{1}{W(1)} = e^{W(1)}$.

\end{thm}
\begin{proof}
Let $G$ be a graph of maximum degree $\Delta \ge 3$ and girth at least 7.\\
We fix $\alpha \coloneqq 1 + \frac{1}{\sqrt{\Delta}}$, $\tau \coloneqq \alpha\Delta$, and $\sigma \coloneqq \frac{1}{2(\alpha^5 - \alpha^3)} + \frac{2}{\Delta}\pth{\frac{1}{\alpha^3 - \alpha}}^2$. Let $K \coloneqq \ceil{\frac{\tau + \sigma}{W(1)} + \sqrt{\Delta}}$.\\
Using the fact that $\alpha^{p+2} - \alpha^{p} \ge 2(\alpha-1) = \frac{2}{\sqrt{\Delta}}$ for all $p \ge 0$, we have that $K \le \frac{\Delta}{W(1)} + \bigO{\sqrt{\Delta}}$.

For every subgraph $H \subseteq G$ and every subset of vertices $U \subseteq V(H)$, let $\cycles_H(U)$ be the set of cycles of $H$ that contain every vertex of $U$. By definition, $\cycles_H \coloneqq \cycles_H(\emptyset)$ is the set of all cycles of $H$. for every subset of cycles $\cycle \subset \cycles_H$, let $\cA(H,\cycle)$ be the set of proper $\Pi$-acyclic $K$-colourings of $H$. The set of acyclic $K$-colouring of $H$ is $\cA(H,\cycles_H)$.

\noindent
For a given vertex $v$, let $\twice_v \coloneqq \bigcup\limits_{\{u_1,u_2\} \in \binom{N(v)}{2}} \cycles_G(\{u_1,u_2\})$ be the set of cycles of $G$ that contain at least two neighbours of $v$. To alleviate the notations, we write for every $\cycle \subseteq \cycles_G$:
\begin{itemize}
	\item $\cycle \minusA v \coloneqq \cycle \setminus \cycles_G(\{v\})$ the set of cycles of $\cycle$ which do not contain $v$.
	\item $\cycle \minusB v \coloneqq \cycle \setminus \twice_v$ the set of cycles of $\cycle$ which contain at most one neighbour of $v$.
\end{itemize}

\begin{claim}\label{fact:cycles}
Let $H$ be a subgraph of $G$, and $\cycle$ a set of cycles of $H$. Let $v \in V(G)$.\\
Then $\cA(H, \cycle \minusA v) \subseteq \cA(H, \cycle \minusB v)$.
\end{claim}
\begin{subproof}
    Since any cycle containing $v$ also contains two neighbours of $v$, we have $\cycles_G(\{v\}) \subseteq \twice_v$ and thus $\cycle \minusB v \subseteq \cycle \minusA v$. It follows that a proper $(\cycle \minusA v)$-acyclic colouring is, in particular, a proper $(\cycle \minusB v)$-acyclic colouring, and thus $\cA(H, \cycle \minusA v) \subseteq \cA(H, \cycle \minusB v)$.
\end{subproof}

\noindent
We show with a strong induction that, for every subgraph $H \subseteq G$,
\begin{align}
\label{eq:IH1}
\tag{IH~\ref*{thm:girth7}} \forall v_0 \in V(H), \forall \:\! \cycle \subseteq \cycles_H, \quad \card{\cA(H,\cycle)}\ge \tau \card{\cA(H-v_0, \cycle \minusB v_0)}.
\end{align}
\noindent
If $V(H)$ is empty, \eqref{eq:IH1} is trivially true. Suppose $V(H) \neq \emptyset$ and let $v_0 \in V(H)$, $\cycle \subseteq \cycles_H$.\\
By induction, assume \eqref{eq:IH1} is true for every strict subgraph $H' \subset H$.

Consider the set $\cA(H, \cycle \minusB v_0)$, which contains the proper colourings of $H$ such that any cycle of $\cycle$ which happens to be bicoloured must contain at least two neighbours of $v_0$, i.e. belongs to $\cycle \cap \twice_{v_0}$. Let $\cF \coloneqq \cA(H, \cycle \minusB v_0) \setminus \cA(H,\cycle)$ be the set of \emph{flawed colourings}, for which at least one cycle of $\cycle \cap \twice_{v_0}$ is indeed bicoloured. By definition, we have $\card{\cA(H,\cycle)} = \card{\cA(H,\cycle \minusB v_0)} - \card{\cF}$.

\noindent
To complete the proof of the induction, we show the following inequalities.
\begin{enumerate}[label=(\roman*)]
	\item $\card{\cA(H,\cycle \minusB v_0)}  \ge \; (\tau + \sigma) \card{\cA(H-v_0, \cycle \minusB v_0)}$.
	\item $\card{\cF} \le \; \sigma \card{\cA(H-v_0, \cycle \minusB v_0)}$.
\end{enumerate}

\begin{subproof}[Proof of Inequality (i)]

For a colouring $c \in \cA(H-v_0,\cycle \minusB v_0)$, let $L_c \coloneqq [K] \setminus c(N(v_0))$ be the list of colours that are not present in the neighbourhood of $v_0$, and let $\ell_c \coloneqq \card{L_c}$ be the size of this list. When extending $c$ to $v_0$ by giving it a colour of $L_c$, if a bicoloured cycle is created, it must  contain $v_0$ and is therefore in $\twice_{v_0}$, so these extensions of $c$ belong to $\cA(H,\cycle \minusB v_0)$. Therefore, the \emph{extensions} of $\cA(H-v_0,\cycle \minusB v_0)$ thus obtained are exactly the colourings of $\cA(H, \cycle \minusB v_0)$. Let $\boldc$ be a uniformly random colouring from $\cA(H-v_0,\cycle \minusB v_0)$. It follows that
\begin{align*}
\card{\cA(H,\cycle \minusB v_0)} & = \sum_{c \in \cA(H-v_0,\cycle \minusB v_0)} \ell_c \\
	&= \card{\cA(H-v_0,\cycle \minusB v_0)} \sum_{c \in \cA(H-v_0,\cycle \minusB v_0)} \frac{\ell_c}{\card{\cA(H-v_0, \cycle \minusB v_0)}}  \\
	&= \card{\cA(H-v_0,\cycle \minusB v_0)}\;\esp{\ell_\boldc}.
\end{align*}
\noindent
Inequality $(i)$ is therefore equivalent to $\esp{\ell_\boldc} \ge \tau + \sigma$.

For a vertex $u \in N(v_0)$ and a colouring $c \in \cA(H-v_0,\cycle \minusB v_0)$, let $L_c^\cA(u)$ be the set of available colours for $u$ such that recolouring $c(u)$ still yields a colouring in $\cA(H-v_0, \cycle \minusB v_0)$. Observe that if we redistribute the colour $\boldc(u)$ uniformly at random from $L_{\boldc}^\cA(u)$, $\boldc$ remains uniformly distributed in $\cA(H-v_0,\cycle \minusB v_0)$, and the lists $L_{\boldc}^\cA(w)$ remain unaffected for all $w\in N(v_0)\setminus \{u\}$.

Hence we are in the setting of Lemma~\ref{lem:coupon}, where we take $k \coloneqq K$ and $t \coloneqq \sqrt{\Delta}$; observe that $\frac{t+1}{t} = \alpha$. Let us write $N(v_0)=\{u_1, \ldots, u_d\}$ where $d=\deg(v)\le \Delta$, and let $\boldL_i \coloneqq L_\boldc^\cA(u_i)$ for all $i$. Then, if we redistribute simultaneously $\boldc(u_i)$ uniformly at random from $\boldL_i$ for all $i$, $\boldc$ remains uniformly distributed, and we infer that $\esp{|\boldL|} = \esp{\ell_\boldc}$. Let us consider the probability that $|L_\boldc(u)| \le t$ for some $u\in N(v_0)$. We have

\begin{align*}
\pr{\ell_\boldc^\cA(u) \le t} & = \; \frac{\card{\set{c \in \cA(H-v_0,\cycle \minusB v_0) : \ell_c^\cA(u) \le t}}}{\card{\cA(H-v_0,\cycle \minusB v_0)}} \\
    & \le \frac{t \card{\cA(H \setminus \set{u,v_0},(\cycle \minusB v_0) \minusA u)}}{\card{\cA(H-v_0,\cycle \minusB v_0)}} & \textrm{by definition;} \\
	& \le \frac{t \card{\cA(H \setminus \set{u,v_0},(\cycle \minusB v_0) \minusA u)}}{\tau \card{\cA(H \setminus \set{u,v_0},(\cycle \minusB v_0) \minusB u)}} & \textrm{by \eqref{eq:IH1};} \\
	& \le \frac{t}{\tau} & \textrm{by Claim~\ref{fact:cycles}.}
\end{align*}

\noindent
By linearity of the expectation, we obtain $\esp{\boldX} \le \Delta \frac{t}{\tau} \le \sqrt{\Delta}$, and thus
\begin{equation}\label{eq:espk0}
 k - \esp{\boldX} \ge \frac{\tau + \sigma}{W(1)} + \sqrt{\Delta} - \sqrt{\Delta} = \frac{\tau + \sigma}{W(1)}.
\end{equation}

Therefore

\begin{align*}
	\esp{\ell_{\boldc}} &= \esp{|\boldL|} \ge (k-\esp{\boldX}) \exp \pth{-\frac{\alpha\Delta}{k - \esp{\boldX}}}& \textrm{by  Lemma~\ref{lem:coupon};}\\
	& \ge \frac{\tau+\sigma}{W(1)} \exp\pth{-\frac{\tau}{\tau+\sigma}W(1)} & \textrm{by  \eqref{eq:espk0};}\\
	& > \tau + \sigma & \textrm{using $\frac{1}{W(1)} = \exp(W(1))$.}\\
\end{align*}

\noindent
We thus conclude the proof of Inequality (i).
\end{subproof}

Recall that $\cF \coloneqq \cA(H, \cycle \minusB v_0) \setminus \cA(H, \cycle)$. If a colouring belongs to $\cF$, it induces a bicoloured cycle among $\cycle \cap \twice_{v_0}$. We partition $\cycle \cap \twice_{v_0}$ into two subsets;
\begin{itemize}
	\item $\twiceA$: the cycles of $\cycle \cap \twice_{v_0}$ that contain $v_0$;
	\item $\twiceB$: the cycles of $\cycle \cap \twice_{v_0}$ that do not contain $v_0$.
\end{itemize}

Moreover, for $\ell \ge 4$, let $\twiceA_{2\ell}$ and $\twiceB_{2\ell}$ be the cycles of length $2\ell$ in $\twiceA$ and $\twiceB$ respectively.

\begin{claim}\label{claim:twiceA}
    Let $\ell \ge 4$. Then $\card{\twiceA_{2\ell}} \le \frac{1}{2}\Delta^{2\ell - 3}$.  
\end{claim}
\begin{subproof}
There are at most $\Delta^{2\ell-3}$ paths of length $2\ell - 3$ starting from $v_0$. For each of these paths, there is at most one way to close it into a $2\ell$-cycle since $G$ has a girth of at least $7$. Each cycle is counted twice with this enumeration. Hence, we divide the total by $2$ and obtain $\card{\twiceA_{2\ell}} \le \frac{1}{2}\Delta^{2\ell - 3}$.
\end{subproof}

\begin{claim}\label{claim:twiceB}
    Let $\ell \ge 4$. Then $\card{\twiceB_{2\ell}} \le \frac{\ell - 4}{2}\Delta^{2\ell - 4}$.
\end{claim}
\begin{subproof}
A cycle in $\twiceB_{2\ell}$ contains at least two neighbours of $v_0$ (but not $v_0$ itself): there are at most $\binom{\Delta}{2}$ pairs of such neighbours. For a pair $\{u_1,u_2\}$, we count the number of $2\ell$-cycles containing $u_1$ and $u_2$ but not $v_0$. For $5 \le i \le \ell$, there are at most $\Delta^{i-3}$ paths of length~$i$ from $u_1$ to $u_2$, and at most $\Delta^{2\ell - i - 3}$ paths of length $2\ell- i$ from $u_2$ back to $u_1$. Therefore, $\card{\twiceB_{2\ell}} \le \binom{\Delta}{2} \sum_{i=5}^\ell \Delta^{i-3}\Delta^{2\ell-i-3} \le \frac{\ell-4}{2} \Delta^{2\ell-4}$. 
\end{subproof}
We may now prove Inequality (ii).

\begin{subproof}[Proof of Inequality (ii)]
For a cycle $C \in \cycle \cap \twice_{v_0}$, let $\cF_c \subseteq \cF$ be the subset of colourings in which the cycle $C$ is bicoloured. We have $\card{\cF} \le \sum\limits_{C \in \twiceA} \card{\cF_C} + \sum\limits_{C \in \twiceB} \card{\cF_C}$.

Consider any even cycle $C = (v_0, v_1, \dots, v_{2\ell-1}) \in \twiceA_{2\ell}$. As detailed in point (ii) of the proof of Theorem~\ref{thm:main}, we have an injection from $\cF_C$ to $\cA(H\setminus\{v_0,v_1,\dots, v_{2\ell-3}\}, (\cycle \minusB v_0) \minusA v_1 \minusA \dots \minusA v_{2\ell-3})$, which is itself contained in $\cA(H\setminus \{v_0,v_1,\dots,v_{2\ell-3}\}, \cycle \minusB v_0 \minusB v_1 \minusB \dots \minusB v_{2\ell - 3})$ by Claim~\ref{fact:cycles}. Finally, by applying 
\eqref{eq:IH1} $2\ell-3$ times, we have $\card{\cF_C} \le \frac{1}{\tau^{2\ell - 3}} \card{\cA(H-v_0, \cycle \minusB v_0)}$. Summing over all even cycles in $\twiceA$, we obtain
\begin{align*}
    \sum\limits_{C \in \twiceA} \card{\cF_C} &\le \sum\limits_{\ell \ge 4}  \sum\limits_{C \in \twiceA_{2\ell}} \frac{1}{\tau^{2\ell - 3}}\card{\cA(H-v_0,\cycle \minusB v_0)}\\
    &\le \sum\limits_{\ell \ge 4} \frac{\Delta^{2\ell - 3}}{2}  \frac{1}{\tau^{2\ell - 3}}\card{\cA(H-v_0,\cycle \minusB v_0)} &\textrm{by Claim~\ref{claim:twiceA}};\\
    &\le \frac{1}{2(\alpha^5 - \alpha^3)} \card{\cA(H-v_0, \cycle \minusB v_0)}.
\end{align*}

Likewise, for every even cycle $C = (u_0, u_1, \dots, u_{2\ell-1}) \in \twiceB_{2\ell}$, we have $\card{\cF_C} \le \frac{1}{\tau^{2\ell - 2}} \card{\cA(H, \cycle)}$ (here we apply \eqref{eq:IH1} $2\ell-2$ times as $v_0 \notin C$, by definition). Using Claim~\ref{fact:cycles} together with the inequality $\card{\cA(H,\cycle)} \le K\card{\cA(H-v_0, \cycle \minusA v_0)}$, we obtain $\card{\cF_C} \le \frac{K}{\tau^{2\ell - 2}}\card{\cA(H-v_0, \cycle \minusB v_0)}$. One can check that $K\le 4\Delta$ for all values of $\Delta \ge 3$. Summing over all even cycles in $\twiceB$, we obtain
\begin{align*}
    \sum\limits_{C \in \twiceB} \card{\cF_C} &\le \sum\limits_{\ell \ge 4}  \sum\limits_{C \in \twiceB_{2\ell}} \frac{4\Delta}{\tau^{2\ell - 2}}\card{\cA(H-v_0,\cycle \minusB v_0)}\\
    &\le \sum\limits_{\ell \ge 4} \frac{\ell-4}{2} \Delta^{2\ell - 4} \frac{4\Delta}{\tau^{2\ell - 2}}\card{\cA(H-v_0,\cycle \minusB v_0)}, &\textrm{by Claim~\ref{claim:twiceB}};\\
    &\le \frac{2}{\Delta}\pth{\frac{1}{\alpha^3 - \alpha}}^2 \card{\cA(H-v_0, \cycle \minusB v_0)}, &\textrm{using }\sum\limits_{i \ge 0}(i+1)x^i = \pth{\sum\limits_{i\ge 0} x^i}^2.
\end{align*}

\noindent
We therefore conclude the proof of Inequality (ii):
\begin{align*}
\card{\cF} &\le \sum\limits_{C \in \twiceA} \card{\cF_C} + \sum\limits_{C \in \twiceB} \card{\cF_C} \le \sigma \card{\cA(H-v_0, \cycle \minusB v_0)}. \qedhere
\end{align*}

\end{subproof}

\noindent
Using Inequalities (i) and (ii), we have $\card{\cA(H, \cycle)} \ge \tau \card{\cA(H-v_0, \cycle \minusB v_0)}$, which ends the proof of the induction.

An iterative application of \eqref{eq:IH1} to all $v_0\in V(G)$ implies that $\card{\cA(G, \cycles_G)} \ge \tau^{\card{V(G)}}$, therefore an acyclic $K$-colouring of $G$ exists.
\end{proof}

\section{Concluding remarks and open problems}
\label{sec:conclusion}
\subsection{Graphs with limited overlaps in neighbourhoods}

Given a graph $G$ and an integer $r\ge 1$, let $\Delta^{(r)}(G)$ denote the maximum $s$ such that $K_{r,s} \subseteq G$. In other words, $\Delta^{(r)}(G)$ is the size of the largest common neighbourhood of $r$ distinct vertices. In particular, $\Delta^{(1)}(G)$ is precisely the maximum degree of $G$, and $\coDelta(G)$ is the maximum codegree of $G$.

We recall that for fixed $t$, we have $a(d,K_{2,t}) = \bigO{\sqrt{t} \cdot d}$ by Corollary~\ref{cor:1acyclic}. This was, in fact, shown in the works of Alon \emph{et al.}~\cite{AMR91} without the assumption that $t$ is fixed.

\begin{thm}[Alon et al, 1991]\label{thm:codeg-upper}
    Let $G$ be a graph. Then there exists a constant $\beta$ such that $$\chi_a(G) \le \beta \cdot \Delta(G) \cdot \sqrt{\coDelta(G)}.$$
\end{thm}

In Section~\ref{sec:general}, we considered the obstruction $K_{3,t}$ for a fixed $t$ and showed that $a(d,K_{3,t}) = \bigO{t^{1/4} \cdot d^{5/4}}$ (Corollary~\ref{cor:2acyclic}). Our proof did not rely on the assumption that $t$ is fixed, as we used explicit upper bounds throughout. Therefore, a slightly stronger result actually holds.

\begin{thm}\label{thm:cocodeg-upper}
    Let $G$ be a graph. Then there exists a constant $\beta$ such that $$\chi_a(G) \le \beta \cdot \Delta(G)^{5/4} \cdot \sqrt[4]{\cocoDelta(G)}.$$
\end{thm}

Our analysis of random graphs in Section~\ref{sec:lowerbounds} yields lower bounds confirming that Theorems~\ref{thm:codeg-upper} and~\ref{thm:cocodeg-upper} are tight up to a $\polylog(\Delta)$ factor, provided that $\coDelta(G)$ and $\cocoDelta(G)$ increase reasonably fast in terms of $\Delta$.

\begin{thm}\label{thm:codeg-lower}
    Let $0 < \alpha \le \frac{2}{3}$. For infinitely many values of $\Delta$, there exists a graph of maximum degree $\Delta$ and maximum codegree $\coDelta(G) = \bigTheta{\Delta^{\alpha}}$ such that $$\chi_a(G) = \bigOmega{\frac{\Delta\cdot \sqrt{\coDelta(G)}}{\sqrt{\log \Delta}}}.$$
\end{thm}

\begin{proof}
    Let $n$ be a (large) integer, and set $p\coloneqq n^{-\frac{1-\alpha}{2-\alpha}}$, thus $\sqrt{\frac{\log n}{n}}\ll p \ll \pth{\frac{\log n}{n}}^{1/4}$. Sample $G \gets G(n,n,p)$. With high probability, the following hold simultaneously:
    \begin{enumerate}[label=(\roman*)]
        \item $\chi_a(G) = \bigOmega{n^{3/2}p^2/\sqrt{\log n}}$ (Corollary~\ref{cor:reformulation}).
        \item $\Delta(G) = \bigTheta{np}$ (Lemma~\ref{lem:randomgraph-degrees}).
        \item $\coDelta(G) =\bigTheta{np^2}$ (similar proof to that of Lemma~\ref{lem:randomgraph-degrees}).
    \end{enumerate}
    Let us assume that (i), (ii), and (iii) hold. We have chosen $p$ such that $np^2 = (np)^\alpha$, therefore $$\coDelta(G) = \bigTheta{\Delta(G)^{\alpha}}.$$
    Furthermore, we observe that $n^{3/2}p^2 = (np)\cdot \sqrt{np^2} = \bigTheta{\Delta(G) \cdot \sqrt{\coDelta(G)}}$. Also, since $\sqrt{n\log n} \ll np < n$, we have $\log(n) = \bigTheta{\log \Delta(G)}$. Therefore,
    $$\chi_a(G) = \bigOmega{\frac{\Delta(G) \cdot \sqrt{\coDelta(G)}}{\sqrt{\log \Delta}}}.$$
\end{proof}

\begin{thm}\label{thm:cocodeg-lower}
    Let $0 < \alpha \le \frac{1}{3}$. For infinitely many values of $\Delta$, there exists a graph of maximum degree $\Delta$ and $\cocoDelta(G) = \bigTheta{\Delta^{\alpha}}$ such that $$\chi_a(G) = \bigOmega{\frac{\Delta^{5/4}\cdot \sqrt[4]{\cocoDelta(G)}}{\sqrt{\log \Delta}}}.$$
\end{thm}

\begin{proof}
    The proof is similar to that of Theorem~\ref{thm:codeg-lower}. Set $p \coloneqq n^{-\frac{1-\alpha}{3-\alpha}}$ and sample $G\gets G(n,n,p)$. With high probability, we have simultaneously $\chi_a(G) = \bigOmega{n^{3/2}p^2/\sqrt{\log n}}$, $\Delta(G) = \bigTheta{np}$, and $\cocoDelta(G) = \bigTheta{np^3} = \bigTheta{\Delta(G)^\alpha}$. Thus
    \[\chi_a(G) = \bigOmega{\Delta(G)^{5/4}\cdot\sqrt[4]{\cocoDelta(G)}/\sqrt{\log \Delta(G)}}.\]
\end{proof}

\subsection{Extended analysis of random graphs}

In Section~\ref{sec:lowerbounds}, we only provided a lower bound on the acyclic chromatic number of a random bipartite graph $G \gets G(n,n,p)$ for $p$ within a certain range, as that was sufficient for our needs. Here we extend our analysis to the whole range of possible values for $p$.
%
Using results from our work and the literature, we completely determine the acyclic chromatic number of $G\gets G(n,p)$ and $G\gets G(n,n,p)$ up to a $\polylog(n)$ factor.

\begin{thm}\label{thm:randomgraph-generalbounds}
    Let $n$ be a (large) integer, and $p = p(n) \in (0,1)$. Sample $G \gets G(n,p)$ or $G\gets G(n,n,p)$. With high probability,
    \begin{equation*}
\chi_a(G) = \begin{cases}
\softTheta{np + 1} &\text{if }p \le 9\sqrt{\frac{\log n}{n}};\\
\softTheta{n^{3/2}p^2} &\text{if }9\sqrt{\frac{\log n}{n}} \le p \le 3\pth{\frac{\log n}{n}}^{1/4};\\
\bigTheta{n} &\text{if }p\ge 3\pth{\frac{\log n}{n}}^{1/4}.
\end{cases}
\end{equation*}
\end{thm}
\begin{proof}
    For the first case, we use the lower bound $\chi_a(G) > \frac{ad(G)}{2} + 1$ given by Fertin \emph{et al.}~\cite{fertin2003acyclic} together with the fact that $\ad(G)=\softTheta{np + 1}$ w.h.p. to obtain that w.h.p. $\chi_a(G) = \softOmega{np + 1}$. Since $p = \bigO{ \sqrt{\log(n)/n}}$, we have $\coDelta(G) = \bigO{\log n}$ w.h.p., which together with Theorem~\ref{thm:codeg-upper} implies that w.h.p. $\chi_a(G) = \softO{np + 1}$.
    
    For the second case, Corollary~\ref{cor:reformulation} states that w.h.p. $\chi_a(G) = \bigOmega{n^{3/2}p^2/ \sqrt{\,\log n}}$. We apply Theorem~\ref{thm:codeg-upper} to obtain that w.h.p. $\chi_a(G) = \bigO{\Delta(G) \sqrt{\coDelta(G)}} = \bigTheta{np \sqrt{np^2}}$.

    Finally, the third case can be proven by slightly modifying the argument in~\cite{AMR91}, and observing that $n$ colours always suffice for an acyclic colouring of $G$.
\end{proof}
In comparison, for all $p = p(n) \in (0,1 - \eps)$, w.h.p. $\chi(G(n,p)) = \bigTheta{\frac{np}{\log (np)} + 1}$ (see e.g.~\cite{shamir1987sharp}).

\subsection{Open problems}

Our work illustrates the rich extremal behaviour of the acyclic chromatic number of $\F$-free graphs. While we were able to entirely characterise families $\F$ that force a sublinear bound on $a(d,\F)$ --- those that contain at least one forest ---, we do not yet have a complete understanding of that regime. Many other questions remain open for the other possible regimes. 


In order to bound the acyclic chromatic number of graphs excluding a tree as a subgraph (Corollary~\ref{cor:forest}), we have proven a more general result concerning the acyclic chromatic number of graphs with bounded degeneracy (Theorem~\ref{thm:degenerate}). We do not have a better lower bound than $\sqrt{\frac{\Delta+1}{2}}$, which results from the fact that $\chi_a(K_n^{1/2}) > \sqrt{n/2}$ (cf.~\cite{wood2005acyclic}).

\begin{question}
    What is the maximum value of $\chi_a(G)$ over all $t$-degenerate graphs $G$ of maximum degree $\Delta$?
\end{question}

Motivated by the bound $a(d,K_{2,t}) = \bigO{\sqrt{t} \cdot d}$ of Alon \emph{et al.}~\cite{AMR91}, we sought to exhibit a large family of graphs $F$ such that $a(d,F)$ is linear in $d$. Our description is certainly incomplete; it would be very interesting to present an even larger family.

\begin{question}
    For which graphs $F$ does it hold that $a(d,F) = \bigO[F]{d}$?
\end{question}

In Section~\ref{sec:lowerbounds}, we have proven that for a fixed $t$, the upper bound $a(d,K_{3,t}) = \bigO{t^{1/4} d^{5/4}}$ is tight up to a factor $d^{\bigO{1/t}}$ (Corollary~\ref{cor:K_3t-free}). This is not satisfying for small values of $t$; for instance the case $t=3$ remains largely open, as we do not have better bounds than $\frac{d}{2} + 1 < a(d,K_{3,3}) = \bigO{d^{5/4}}$.

\begin{question}
    What is the order of magnitude of $a(d,K_{3,3})$? 
\end{question}

In Section~\ref{sec:cycle}, we focused specifically on even cycle obstructions and proved a $2\Delta + o(\Delta)$ upper bound holds. We then showed that this constant $2$ can be replaced by $1.7633$ for graphs of girth $7$. Naturally, one may ask how small this constant may be as we increase the girth. Note that this constant cannot be smaller than $\frac{1}{2}$, as illustrated by the proof of Proposition~\ref{prop:linearlowerbound:cycle}.

\begin{question}
    For $\ell \ge 4$, let $\mathscr{C}_\ell \coloneqq \{C_3, C_4, \dots, C_\ell\}$. In other words, a graph is $\mathscr{C}_\ell$-free if it has girth $\ell + 1$. Is it true that $a(d,\mathscr{C}_\ell) \le (1 + \varepsilon_\ell)d$, with $\varepsilon_\ell \rightarrow 0$ as $\ell \rightarrow \infty$?
\end{question}

Finally, we address a more general question inspired by a similar conjecture of Erd\H{o}s and Simonovits, which states that for every rational $\alpha \in [1,2]$, there exists a graph $F_\alpha$ such that $\ex(n,F_\alpha) = \bigTheta{n^\alpha}$.
For every single-graph obstruction $F$, we define $\alpha_F \coloneqq \lim_{d\to \infty} \limits \frac{\log a(d,F)}{\log d}$. We have seen that there are infinitely many possible values in the range $[1,4/3]$ for $\alpha_F$. For instance, $\alpha_{\hammock{K^{1/2}_{s,t}}} \underset{t\to\infty}{\to} \frac{5}{4} - \frac{1}{4s}$ by Corollary~\ref{cor:subdivision-upperbound} and Corollary~\ref{cor:subdivision-lowerbound}. We may ask whether one can cover the whole range $[1,4/3]$ with the values of $\alpha_F$.

\begin{question}
\label{q:exponent}
    For every rational $\alpha \in [1,4/3],$ is there a graph $F$ such that $\alpha_F=\alpha$?
\end{question}

\bibliographystyle{abbrv}
\bibliography{acyclic}

\end{document}